\documentclass[10pt,amssymb]{article}

\usepackage{amsthm,amssymb,amsmath, graphicx}
\usepackage[all]{xy}

\newcommand{\calK}{\mathcal{K}}
\newcommand{\calH}{\mathcal{H}}
\newcommand{\frakp}{\mathfrak{p}}
\newcommand{\olM}{\overline{M}}

\newtheorem{thm}{Theorem}[section]    
\newtheorem{lemma}[thm]{Lemma}          
\theoremstyle{definition}
\newtheorem{corollary}[thm]{Corollary}

\begin{document}

\title{Extended Hodge Theory for Fibred Cusp Manifolds}
\author{E. Hunsicker
\\ Loughborough University\\
E.Hunsicker@lboro.ac.uk}
\date{\today}							

\maketitle

\begin{abstract}    
For a particular class of pseudo manifolds, we show that the intersection cohomology groups for any perversity may
be naturally represented by extended weighted $L^2$ harmonic forms for a complete metric on the regular stratum with respect to 
some weight determined by the perversity.  Extended weighted $L^2$ harmonic forms are harmonic forms that are almost
in the given weighted $L^2$ space for the metric in question, but not quite.  This result is akin to the representation of absolute
and relative cohomology groups for a manifold with boundary by extended harmonic forms on the associated manifold with cylindrical ends.  
In analogy with that setting, in the unweighted $L^2$ case, the boundary values of the extended harmonic forms define a Lagrangian
splitting of the boundary space in the long exact sequence relating upper and lower middle perversity intersection
cohomology groups.  

\end{abstract}


\section{Introduction}
The Hodge theorem for compact smooth manifolds was a major breakthrough connecting geometry, topology and analysis 
on smooth compact manifolds.  It says that the space of harmonic forms over a compact Riemannian manifold, $(M,g)$,
(an analytic quantity) is naturally isomorphic to the deRham cohomology of $M$ (a topological quantity).  Further, 
it can be refined to say that when $g$ is a K\"ahler metric, the space of harmonic forms breaks down by holomorphic/antiholomorphic
bidegree to form a Hodge diamond of spaces.  This decomposition, together with maps among the 
pieces given by the Hodge star operator and complex conjugation, is called a Hodge structure.  The isomorphism with deRham cohomology also then endows these topologically defined spaces with a Hodge structure.

In addition to the Hodge theorem for compact smooth $M$, when $M$ is oriented, there is also a natural bilinear pairing on smooth forms over $M$
which descends to a nondegenerate pairing on the deRham cohomology, $H^*(M)$, called the signature pairing.  
The signature of this pairing is called simply
the signature of $M$.  This can also be realised as a pairing on the space of harmonic forms, where it is also nondegenerate,
and gives the same signature through the Hodge isomorphism.  

Over the past several years, many mathematicians have worked to generalise these results to settings in which
the compact manifold, $M$, is replaced either with a noncompact manifold or with a singular space.  This work has 
gone on from both the topological and analytic sides.  On the topological side, the deRham cohomology (and dual cellular and other homologies) needed to be replaced by more general cohomologies and homologies.  On the analytic side, tools needed to be developed for studying elliptic operators over noncompact or incomplete manifolds endowed with various classes of metrics.  
The first work in this direction is due to Connor in 1956, who studied manifolds with boundary and the relationship between 
absolute and relative
cohomology and harmonic forms with either Neumann or Dirichlet boundary conditions, \cite{Co}.  In 1975, \cite{aps1}, Atiyah, Patodi and Singer
studied signatures on manifolds with boundary, and also Hodge results for a noncompact manifold with cylindrical ends.  These
results again related to relative and absolute cohomology.

In the 1980's, the development of intersection homology and cohomology on pseudomanifolds critically involved the collaboration
of an analyst, Cheeger, and two topologists, Goresky and MacPherson.  Together, they defined this new (co)homology theory
on singular spaces that recaptures Poincar\'e duality and various other properties of homology and cohomology on smooth
compact manifolds, and proved its isomorphism to the $L^2$ cohomology of manifolds with conical and iterated conical metrics,
under a topological assumption about the links of singular strata in the pseudomanifold.  Following on from this, in 1990, Saper
and Stern proved Zucker's conjecture, that the $L^2$ cohomology of a Hermitian locally symmetric space under its
natural Bergman metric is isomorphic to the middle perversity intersection cohomology of its reduced Borel compactification.

Since then, many authors have worked on Hodge theorems relating harmonic forms or $L^2$ cohomology on noncompact
manifolds with various natural classes of metrics to intersection cohomology, and in the setting of incomplete manifolds,
understanding the relationship between intersection cohomology and boundary conditions for the Laplace operator on 
differential forms.  Just a few such papers are \cite{MaPh}, \cite{Ma}, \cite{HM}, \cite{ALMP1}, \cite{Car1},
\cite{Car2}, \cite{Car3}.  The most relevant to this paper are \cite{HHM}, in which the author of the present paper and her collaborators studied
$L^2$ harmonic forms on fibred cusp manifolds and their relationship to middle perversity intersection cohomology groups, 
and \cite{JM}, in which harmonic forms on fibred cusp manifolds satisfying the same geometric restriction as 
in this paper are found which represent relative and absolute cohomology groups.
At the same time, many authors have been involved with the development of new analytic tools
for the study of singular and noncompact spaces.  A very incomplete list of relevent works on pseudodifferential techniques are 
\cite{Me-aps}, \cite{MaMe}, \cite{Va}, \cite{GH1}, \cite{GHAalto}, and \cite{Lesch}.

In this paper we study the Hodge structure and signature pairing for pseudomanifolds $X$ with one smooth singular stratum $B$.  
Let $Y\stackrel{\phi}{\to}B$ denote the link bundle of $B$ and $F$ be the link.  Let $M = X-B$.  It is also useful to define the blowup of $X$ along
$B$, which is a manifold with boundary $\olM:= M \cup [0,1) \times Y$, where we fix an identification $U \cong (0,1) \times Y$,
of the regular stratum of a normal neighborhood of $B$ in $X$.  Let $x$ denote the variable on $(0,1)$, and extend it as a smooth function $>1$ on the rest of $M$.  We can think of $U$ 
as ``the end" of $M$, and we obtain $\olM$ by adding an additional copy of $Y$ at $x=0$.

On a pseudomanifold, $X$, the appropriate cohomology theory to study for Hodge results and signatures is the intersection
cohomology.  Intersection cohomology is a family of cohomologies on $X$ parametrised by a function called the perversity.
In the original definition of intersection (co)homology, a perversity is a function
$\frakp:\mathbb{N} \to \mathbb{N}$ satisfying certain restrictions.
Near a point on a codimension $l$ stratum of a pseudomanifold, $X$,
$\frakp(l)$ determines the form of the Poincar\'e lemma.
Thus when there is a single, connected, smooth singular stratum $B \subset X$ of codimension
$f+1$, the only part of the perversity which affects $IH_\frakp^*(X)$, is the value $\frakp(f+1)$.  
We can first simplify notation in this setting by labelling the intersection cohomology
spaces by $p:=\frakp(f+1)$ rather than by $\frakp$.  Further, we will fix notation such that the Poincar\'e lemma
for a cone has the form:
\[
IH^j_p(C(F)) = \left\{ 
\begin{array}{ll}
H^j(F) & j< p \\
0 & j \geq p.
\end{array} \right.
\]

In terms of this convention, standard perversities (those satisfying the original restrictions)
satisfy $0 < p \leq f $. We use an extension of these definitions
in which $p \in \mathbb{R}$.   This does not
give anything dramatically new; when $p \leq 0$,
we get $H^*(X,B)$, whereas when
$p > f$ we get
$H^*(X-B)$. For $p \notin \mathbb{Z}$, we simply get the same thing as for 
the perversity given by the floor function $\lfloor p \rfloor$.  Thus for any $p \in {\mathbb R}$
we fix the notation
\begin{equation}
IH^*_{p}(X,B) = \left\{ \begin{array}{ll}
H^*(X-B) &p > f, \\
IH_p^*(X) & 0< p \leq f, \\
H^*(X, B) & p\leq 0,
\end{array} \right. .
\label{eq:extih}
\end{equation}

This notation will simplify the statements and proofs of results, and in addition, allows us to consider situations in which $B$ is a codimension one stratum, i.e.\ a boundary, together with the case in which $B$ is a singular stratum. In the case where $B$ is a boundary, the link of a point on $B$ is simply a point, so the local calculations corresponding to
$p > 0$ and $p \leq 0$ give absolute and relative cohomologies, respectively. In the remainder of this paper, we will assume that $X$ has no boundary, except in the case that its boundary is $B$.

The metrics we will consider on $M$ are various classes of fibred cusp metrics.  The most general definition of
a fibred cusp metric is that it is $x^2$ times a section of the symmetric product of the $\phi$ cotangent bundle with itself, see \cite{GH1},
for example.  In our calculations, we will be able to restrict to product type fibred
cusp metrics, which are smooth metrics on $M$ which on the end have the form:
\[
g_{fc} = \frac{dx^2}{x^2} + \phi^*ds_B^2 + x^2h.
\]
Here $\phi^*ds_B^2$ represents the lift to $Y$ of a metric on the base, $B$.  
For any choice of smooth horizontal tangent bundle, complementary to the vertical tangent bundle of the fibres, 
we choose a symmetric bilinear form, $h$, on $Y$ that is diagonal with respect to this splitting,  vanishes on 
the horizontal tangent bundle and is positive definite on the vertical tangent bundle.  
Geometrically, such metrics mean that the manifold looks at the end
like a cusp bundle over a compact base.  In particular, these metrics are geodesically complete.  

Note that the term ``product" in the phrase ``product type fibred cusp metric" does not refer to the metric on $Y$,
but rather the fact that $ds_B^2$ and $h$ are independent of $x$.   
In general, we cannot assume that the metric on $Y$ is even locally a product.  This is possible when
$Y$ is a flat bundle with respect to the structure group ${\rm Isom}(F,g_F)$ for some fixed metric $g_F$
on $F$.  When this is possible, and when we take the metric on $Y$ to be a local product, we say that
the metric $g_{fc}$ is a {\em geometrically flat} fibred cusp metric.  Whereas the results below about weighted $L^2$ 
harmonic forms require only that the first order Gauss Bonnet operator is Fredholm with a particular regularity result, the
results for extended
harmonic forms require regularity and Fredholm results for the (second order) Hodge Laplacian.  As has been noted in
\cite{GHAalto}, reflecting a similar condition required for the results in \cite{JM}, this requires the more stringent condition 
that the metric is an ``admissible perturbation" of a geometrically flat fibred cusp metric.  
The background analytic results we require for the proofs are identical for the 
reference metric and the admissible perturbation, thus in the proofs, we can simply consider product
type and geometrically flat fibred cusp metrics.  

The first result in this paper relates weighted $L^2$ harmonic forms for fibred cusp metrics  on 
$M$ to intersection cohomology of $X$ of various perversities.  This theorem is a slight generalisation of the 
results in \cite{HHM}, and the technique of proof is similar.  First we define the notation used in 
the theorem.  The forms we will consider lie in weighted $L^2$ spaces of smooth forms for the metric $g_{fc}$:
\[
x^cL^2\Omega^*(M,g_{fc}):=\{ \eta \in \Omega^*(M) \mid \int_M x^{-2c}|\eta|_g^2 \, dvol_g\}
\]
\[
:=  \{ \eta =x^c\omega \mid \omega \in  L^2\Omega^*(M,g_{fc}) \}.
\]
This space is equipped with a metric.  For $\tau, \eta \in x^cL^2\Omega^*(M,g_{fc})$, we define
\[
<\tau, \eta>_{c}:= \int_M x^{-c}\tau \wedge *\, x^{-c}\eta : = \int_M \tau \wedge *_c \eta,
\]
where $*$ is the standard Hodge star for forms over $M$ with respect to the volume form on $M$
coming from $g_{fc}$ and $*_c=x^{-2c}*$.  We define $\delta_{fc,c}$ to be the formal adjoint of $d$ with respect
to this weighted pairing.  The spaces of harmonic forms we consider are the weighted $L^2$ harmonic
forms, 
\[
\calH_{L^2}^j(M,g_{fc},c):= \{\omega \in x^cL^2\Omega^*(M,g_{fc}) \mid (d+\delta_{fc,c}) \omega =0\},
\]
and the extended weighted harmonic forms:
\[
\mathcal{H}^j_{ext}(M, g_{fc}, c):= \{\omega \in \cap_{\epsilon>0} x^{c-\epsilon} L^2\Omega^*(M,g_{fc}) \mid (d+\delta_{fc,c}) \omega =0\}.
\]
Note that due to the weighting factor, the operator $d+\delta_{fc,c}$ is not the Gauss-Bonnet operator associated to any metric on $M$ except in the case that
$c=0$, in which it is the operator associated to the fibred cusp metric $g_{fc}$.

We can now state our theorems.
\begin{thm}\label{l2hodge}  Let $X$ be a pseudomanifold with a smooth singular stratum, $B$, 
and assume the fibre $F$ of the link bundle over $B$ has dimension $f$.  Endow $M=X - B$ with
any fibred cusp metric $g_{fc}$. 
Then the weighted $L^2$ harmonic forms on $M$ can be interpreted
as:
\[
\mathcal{H}^j_{L^2}(M, g_{fc}, c) \cong {\rm Im}\left(IH^j_{(f/2)-c-\epsilon}(X, B) \to IH^j_{(f/2)-c+\epsilon}(X, B)\right).
\]
Further, there exists a nondegenerate bilinear pairing
\[
\cap_{L^2}: \mathcal{H}^*_{L^2}(M, g_{fc}, c) \otimes \mathcal{H}^{n-*}_{L^2}(M, g_{fc}, -c) \to \mathbb{R}.
\]
When $c=0$, the signature of this pairing is the middle perversity perverse signature on $X$.
\end{thm}
\noindent
Perverse signatures are the natural signatures defined by the intersection pairing in intersection cohomology,
restricted to spaces of the form
\[
{\rm Im}(IH_{f/2-c}^*(X,B) \to IH_{f/2+c}^*(X,B)),
\]
and have been studied in \cite{Hu2} and \cite{FH}.

\medskip
The next theorem relates ``extended" weighted $L^2$ 
harmonic forms for the same metrics to groups coming from intersection cohomology of $X$.  
This theorem is a generalisation of 
Proposition 6.16 from \cite{Me-aps}, which considers the case when $c=0$ and $f=0$, that is,
the fibre $F$ is trivial, and refines the theorem from
\cite{aps1} that says relative and absolute cohomology of a manifold with boundary
may be represented by subspaces of extended harmonic forms on the associated
manifold with infinite cylindrical end.

\begin{thm}\label{exthodge}
Let $X$ be an $n$ dimensional pseudomanifold with a smooth singular stratum, $B$, 
and assume the fibre $F$ of the link bundle over $B$ has dimension $f$.  Endow $M=X - B$ with
a metric,  $g_{fc}$, that is an admissible perturbation of a geometrically flat 
fibred cusp metric.  Then 
the space of extended weighted $L^2$ harmonic forms on $M$ decomposes into
three pieces
\[
\mathcal{H}^j_{ext}(M, g_{fc}, c) = dS_c + \delta_{fc,c}T_c + \mathcal{H}^j_{L^2}(M, g_{fc}, c),
\]
where the sums here are vector space direct sums.
We can naturally 
interpreted this space in terms of intersection cohomology of $X$ through the following isomorphisms:
\[
IH^j_{(f/2)-c-\epsilon}(X, B) \cong  dS_c + \calH_{L^2}^j(M,g_{fc},c) 
\]
and
\[
IH^j_{(f/2)-c+\epsilon}(X, B) \cong \delta_{fc,c}T_c + \calH_{L^2}^j(M,g_{fc},c).
\]
\end{thm}

This theorem justifies the following notation:
\[
\mathcal{IH}^j_{(f/2)-c-\epsilon}(M,g_{fc}) := dS_c + \calH_{L^2}^j(M,g_{fc},c) 
\]
and
\[
\mathcal{IH}^j_{(f/2)-c+\epsilon}(M,g_{fc}) := \delta_{fc,c}T_c + \calH_{L^2}^j(M,g_{fc},c).
\]

We can refine this to the following theorem, which generalises Proposition 6.18 in \cite{Me-aps}, and shows how 
the spaces of extended harmonic forms on $M$ and the intersection cohomology on $X$ fit together:
\begin{thm}\label{les}  Assume $p$ is an integer.
If the link bundle $Y\to B$ of $B$ in $X$ is geometrically flat, then there exists a long exact sequence on cohomology:
\[
\cdots \to H^{j-p-1}(B,H^p(F)) \stackrel{\partial}{\longrightarrow} IH_{p-1}^j(X) \stackrel{inc^*}{\longrightarrow} IH_{p}^j(X) 
\stackrel{r}{\longrightarrow} H^{j-p}(B,H^p(F)) \to \cdots.
\]
If $g_{fc}$ is an admissible perturbation of a geometrically flat 
fibred cusp metric on $M$, then there is in addition a commutative diagram for each $j$ and $p$:
\[
\xymatrix{
\calH^{j-p-1}(B,\calH^p(F)) \ar[r]^{\tilde\partial} \ar[d]^{\cong}& \mathcal{IH}^j_{p-\epsilon}(M,g_{fc}) \ar[r]^{inc^*}  \ar[d]^{\cong}& \mathcal{IH}^j_{p+\epsilon}(M,g_{fc}) 
\ar[r]^{r}  \ar[d]^{\cong}&\calH^{j-p}(B,\calH^p(F)) \ar[d]^{\cong} \\
H^{j-p-1}(B,H^p(F)) \ar[r]^{\qquad \partial} & IH_{p-1}^j(X) \ar[r]^{inc^*} & IH_{p}^j(X) 
\ar[r]^{r \qquad} &H^{j-p}(B,H^p(F)) .
}
\]
\end{thm}
\noindent
If $p$ is not an integer, then the diagram collapses, with the spaces on the far right and left equaling $\{0\}$ and the
spaces in the middle all being isomorphic, and also isomorphic to $\calH^j_{L^2}(M,g_{fc})$.


The extended spaces of harmonic forms satisfy some interesting additional properties, which are summarised in 
the next theorem.
\begin{thm}\label{1.6}
If $g_{fc}$ is an admissible perturbation of a geometrically flat fibred cusp metric on $M$,
then 
the following form of Poincar\'e duality holds on extended harmonic forms:
the operator $*_c = x^{-2c}*$, where $*$ is the Hodge star operator on $L^2\Omega^*(M,g_{fc})$, gives
an isomorphism
\begin{equation}\label{poincare}
*_c: \mathcal{IH}^j_{(f/2)-c-\epsilon}(M,g_{fc}) \to \mathcal{IH}^{n-j}_{(f/2)+c+\epsilon}(M,g_{fc}).
\end{equation}
Further, when $f/2-c$ is an integer, there are natural boundary data maps,
\[
{\rm bd_T}:\delta_{fc,c}{\rm T}_c \to \calH^*(B,\calH^{f/2-c}(F)), \qquad {\rm bd_S}:d{\rm S}_c \to \calH^*(B,\calH^{f/2-c}(F)),
\]
whose images are complementary orthogonal subspaces of $H^*(B,H^{f/2-c}(F))$.  Thus in particular, when $f$ is even and $c=0$,
\[
{\rm Im(bd_T)} \oplus {\rm Im(bd_S)} \subset H^*(B,H^{f/2}(F)) \oplus H^*(B,H^{f/2}(F))
\]
is a Lagrangian subspace under the boundary pairing 
\[
iB((u_1,u_2),(v_1,v_2)) = \langle u_1,v_2\rangle_Y - \langle u_2,v_1 \rangle_Y.
\]

\end{thm}

The results in this paper are interesting in their own right, but are also interesting in their relationship to 
other recent work about cohomology and harmonic forms on pseudomanifolds.  The Lagrangian subspace in Theorem \ref{1.6} is analogous to the structure
that permits the definition of the Atiyah-Patodi-Singer boundary condition for the Hodge Laplacian on manifolds with 
boundary.  The Lagrangian subspace defined by the boundary data of extended harmonic forms
in the fibred cusp setting should analogously relate to self-adjoint boundary conditions for the Hodge Laplacian
on a manifold with an edge metric.  It would be interesting to understand how both the APS boundary 
conditions and the conditions coming from Theorem \ref{1.6} relate to the abstract self-adjoint extensions of the Hodge
Laplacian for open, orientable incomplete manifolds identified by Bei in \cite{Bei}.

Additionally, it would be interesting to understand how both Bei's work and the splitting
arising in Theorem \ref{1.6} relate to the refinement of intersection cohomology in the
case of spaces that are not Witt, but possess a particular Lagrangian structure on the bundle of middle
degree link cohomology groups over the singular strata.  Such spaces, sometimes referred to as Cheeger
spaces, and the resulting ``mezzo perversity" intersection cohomology groups that may be defined
on them have been studied by Cheeger in \cite{Ch}, Banagl in \cite{Ba-lag} and by Albin et al in \cite{ALMP1}.
It would be interesting to understand when the Lagrangian structure
given by Theorem \ref{1.6} comes from a Lagrangian structure on the $H^{f/2}(F)$ bundle over $B$, to obtain a 
better understanding of the relationship between extended harmonic
forms for fibred cusp metrics and the associated mezzo perversity cohomology group $IH^*_\mathcal{L}(X)$.
This would also clarify the relationship between the signature defined on Cheeger spaces and the middle
perversity perverse signature studied in \cite{Hu2} and \cite{FH}.

Finally, extended weighted $L^2$ harmonic forms for fibred cusp metrics are equivalent to extended weighted $L^2$ harmonic
forms for the conformally equivalent fibred boundary metrics, but with a shift in weight.  In \cite{BH}, the author and Banagl
use this relationship to prove a Hodge theorem for HI cohomology, defined by Banagl in \cite{Ba-HI}.

The author would like to acknowledge the help of two anonymous reviewers in suggesting improvements to this paper that clarify some points and improve readability.


\section{Background}

As with absolute or relative cohomology, there are many possible complexes of forms that can be used to calculate
$IH^j_p(X,B)$.  One useful complex comes from the weighted $L^2$ forms over $M$ with respect to a fibred cusp metric.
This lemma is proved in \cite{HHM}:
\begin{lemma}\label{wcoho} 
Let $WH^*_{fc}(M,c)$ be the cohomology of the chain complex 
\[
WC^*_{fc}(M,c):=\{\omega \in x^{c}L^2\Omega^*(M,g_{fc}) \, \mid \, d\omega \in x^{c}L^2\Omega^*(M,g_{fc}) \}.
\]
Then if $H^{\frac{f}{2}-c}(F)=0$, 
\[
WH^*_{fc}(M, c) \cong IH^*_{\frac{f}{2}-c}(X,B).
\]
\end{lemma}

Note that when $H^{\frac{f}{2}-c}(F) \neq 0$, we get
\[
WH^*_{fc}(M, c-\epsilon) \cong IH^*_{\frac{f}{2}-c}(X,B)
\]
and
\[
WH^*_{fc}(M, c+\epsilon) \cong IH^*_{\frac{f}{2}-c-1}(X,B).
\]
Further, there is a natural inclusion of complexes $WC^*_{fc}(M,c+\epsilon) \subset WC^*_{fc}(M,c-\epsilon)$
that induces a natural map on cohomology $IH^*_{p-1}(X,B) \to IH^*_p(X,B)$ for any $p$.

We can choose to replace the smooth $x^cL^2$ forms here with conormal forms in $x^cL^2$.  Specifically,
a form $\omega \in x^cL^2\Omega^*(M,g_{fc})$ is conormal if $P\omega \in x^cL^2\Omega^*(M,g_{fc})$
for any differential operator of the form $P = (x\partial_x)^aP_Y$, where $P_Y$ is any differential operator on $Y$.
The complex of conormal $x^cL^2$ forms over $M$ is:
\[
x^{c}L^2\Omega_{con}^*(M,g_{fc}):= \{ \omega \in x^{c}L^2\Omega^*(M,g_{fc}) \mid \hspace{1in}
\]
\[
\hspace{1in} d_M \omega \in x^{c}L^2\Omega^*(M,g_{fc})
\mbox{ and } P\omega \in x^{c}L^2\Omega^*(M,g_{fc}) \, \forall P \mbox{ as above}\}.
\]
The proof that this complex calculates intersection cohomology is exactly the same as the proof for the smooth complex of 
$x^cL^2$ forms (see eg \cite{HHM}), through checking
that the Poincar\'e lemma for intersection cohomology is satisfied.
This is the complex we will use for the proof of Theorem \ref{l2hodge}.

The proofs in this paper are based on analytic results from \cite{GHAalto}.
This paper reproves and extends some the results in \cite{Va} that are referred to and 
used in \cite{HHM} to prove related Hodge theorems.   The results use the theory of the b-calculus \cite{Me-aps}
and the phi-calculus \cite{MaMe}, and are stated in terms of ``split" sobolev spaces and
split elliptic operators.  For the technical details of these defintions we refer the reader to 
\cite{GHAalto}, but we can roughly explain the ideas as follows.  
First of all, we have to 
complete the space 
$x^{a} L^2\Omega^*(M, g_{fc})$ using the inner product $\langle,\rangle_{fc,a}$.  Call the resulting Hilbert space
$x^{a} L^2(M, g_{fc},\Lambda^*)$.
The ``splitting" in the split sobolev spaces and operators
relates to a decomposition of the space of forms near the
the boundary into fibre harmonic forms
over the base of the boundary fibration (which we think of as forms over $B \times (0,1)$ with values
in a flat bundle, $\mathcal{K}_*$), and the orthogonal complement
of that subspace of forms.   The split sobolev spaces require b-type regularity in the 
``fibre harmonic" components and phi-type regularity in the orthogonal complement.  
 We note that for any metric $g$ on $M$,
\[
x H^1_{\phi}\Omega^*(M, g) \subset x H^1_{split}\Omega^*(M, g) \subset H^1_b \Omega^*(M, g)\subset L^2(M, g,\Lambda^*).\]
This is Equation (6.1) in \cite{GHAalto}.  Note in particular that the $\Omega$ occurring in the first three spaces here is an abuse of notation, and does not indicate smoothness on the interior.

The first analytic result we need based on \cite{GHAalto} is a Fredholm theorem.  
\begin{thm}
For all but a discrete set of $a \in \mathbb{R}$, the map
\[
d+\delta_{fc,c}: x^{a+1} H^1_{split}\Omega^*(M, g_{fc}) \to x^{a} L^2(M, g_{fc},\Lambda^*)
\]
is Fredholm.  
In particular, for some sufficiently small $\epsilon>0$, the operators
are Fredholm for $a=c \pm \epsilon$.  
 \end{thm}
 \begin{proof}
The proof of this theorem is exactly the same as the proof when $c=0$, which is Theorem 1 in \cite{GHAalto}, or Proposition 17 in 
\cite{HHM}, where the required indicial root calculation is given.  The only difference between $d+\delta_{fc,0}$ and $d+\delta_{fc,c}$ is a 0th order
term:
by direct calculation, we get that, up to isomorphism of the $fc$ and $fb$ form bundles,
$d+\delta_{fc,c} =  x^{-1}(d+ \delta_{fb} + xZ_c)$, where $Z_c$ is a bounded zeroth order operator
depending on $c$.
This means that $d+ \delta_{fb} + xZ_c$ is a split elliptic operator of degree 1, and thus by Theorem 13 in \cite{GHAalto}, (also used to prove Theorem 1 in that
paper), is Fredholm as a map from $x^aH^1_{split}\Omega^*(M, g_{fc})$ to
$x^{a} L^2(M, g_{fc},\Lambda^*)$ for all but a discrete set of $a$.  The result follows directly.
 \end{proof}
The second result is a regularity theorem.  To state this one, let $M_{x<1/2}:= \{p \in M : x(p) < 1/2 \} \cong Y \times (0,1/2)$.
We will refer to this set as the {\em end} of $M$.

\begin{thm} \label{regfc} Assume $\epsilon>0$ is sufficiently small, and 
let $\eta \in x^{c-\epsilon} L^2\Omega_{con}^*(M, g_{fc})$. Then $(d+\delta_{fc,c})\eta =\xi \in x^{c+\epsilon} L^2\Omega_{con}^*(M, g_{fc})$
implies 
$\eta = \eta_0 + \eta'$, where $\eta' \in x^{c+\epsilon} L^2\Omega_{con}^*(M, g_{fc})$,
and on the end, 
\[
\eta_0 = x^{c-f/2}\left(x^{(f/2)-c}u_{(f/2)-c} + \frac{dx}{x} \wedge x^{(f/2)-c}v_{(f/2)-c} \right), 
\]
where $u_{(f/2)-c}$ and $v_{(f/2)-c}$ are harmonic forms on $B$ with values in $\mathcal{K}_{(f/2)-c}$.  Further, if 
$\eta \in x^{c+\epsilon} L^2\Omega_{con}^*(M, g_{fc})$, then in fact, 
\[\omega \in x^{c+1+\epsilon} H^\infty_{split}\Omega^*(M, g_{fc})
\subset x^{c+\epsilon}H^\infty_\phi\Omega^*(M, g_{fc}).\]
\end{thm}
\begin{proof}
This generalisation of Theorem 14 from \cite{GHAalto} (which is for the case $c=0$) follows by an identical argument to the proof of that theorem, from the parametrix result for split operators from \cite{GHAalto} (Theorem 12).  It is proved by applying the left parametrix obtained to both sides of the equation $(d+\delta_{fc,c})\eta= \xi$, and from the indicial root calculation for the b-operator obtained in the splitting.  Because the indicial roots here are all simple, the indicial root calculation 
comes down to determining the kernel of $d+\delta_{fc,c}$ (where this is the operator for the product type fibred cusp metric 
that $g_{fc}$ tends to as $x \to 0$) among forms on the cylinder $Y \times [0,1)$ of the form
\[
\eta_0 = x^{c-f/2}\sum_{i=0}^f\left(x^{i}u_{i} + \frac{dx}{x} \wedge x^{i}v_{i} \right),
\]
where $u_{i},v_{i}$ are fixed forms in $\Omega^*(B,\calH^{i}(F))$.
Applying the operator and combining terms with the same power in $x$,
we find that the only solutions are the ones where $i=(f/2)-c$ and 
where $u_{(f/2)-c},v_{(f/2)-c}$ are harmonic.
For more details of the indicial root calculations in the case $c=0$, see also \cite{HHM}.
\end{proof}

If we restrict our consideration to the case when we can use a geometrically flat fibred cusp metric on $M$ (recall there is a topological obstruction to the existence of such a metric),
then any harmonic form on $B$ with values in $\mathcal{K}$ is in fact harmonic on $Y$
\cite{Ba2}.  
We need this restriction in order to get Fredholm and regularity results for 
$\Delta_{fc,c}= (d+\delta_{fc,c})^2$ rather than only for $d+\delta_{fc,c}$.  These theorems and their proofs
are similar to the ones above, so we supress the proofs.  We note, however, that in this geometrically restricted setting they hold for any degree of split Sobolev regularity.  

\begin{thm}\label{fredfclap}
Assume that the metric on $Y$ is geometrically flat.  For all but a discrete set of $a \in \mathbb{R}$, the map
\[
\Delta_{fc,c}: x^{a+2} H^{k+2}_{split}\Omega^*(M, g_{fc}) \to x^{a} H^{k}_{split}\Omega^*(M, g_{fc})
\]
is Fredholm.  
In particular, for some sufficiently small $\epsilon>0$, the operators
are Fredholm for $a=c \pm \epsilon$.  
 \end{thm}

\begin{thm} \label{regfclap} Assume $\epsilon>0$ is sufficiently small, and 
let $\eta \in x^{c+k-\epsilon} L^2\Omega_{con}^*(M, g_{fc})$. Then $\Delta_{fc,c} \eta \in x^{c+k+\epsilon} H^k_{split}\Omega_{con}^*(M, g_{fc})$
implies 
\[\eta = \eta_0 + \eta', \qquad \eta' \in x^{c+k+2+\epsilon} H^{k+2}_{split}\Omega_{con}^*(M, g_{fc}),
\]
and on the end, 
\[
\eta_0 = \left(u_{11}\log x + u_{21}+ \frac{dx}{x} \wedge (u_{12}\log x + u_{22}) \right), 
\]
where $u_{ij}$ are harmonic forms on $Y$ fibre degree $(f/2)-c$.  

\end{thm}

In this restricted setting, the parametrices for $d+\delta_{fc,c}$ and $\Delta_{fc,c}$ can in fact be taken to be the first diagonal
parametrix constructed in \cite{GHAalto}, Section 5.2 .  Near the boundary, this is the lift of the b-parametrix near the boundary for the 
Gauss Bonnet operator (resp., Laplacian) on $B\times (0,1)$ with values in $\calK$ (acting on the fibre harmonic
part of forms) with the
inverse of $d+\delta_{fc,c}$ (resp., $\Delta_{fc,c}$) acting on the orthogonal complement of the fibre harmonic
forms.  

We not that we can obtain the same Fredholm and regularity results for somewhat more general metrics than strict geometrically flat product 
type fibred cusp metrics, which we call {\em admissible perturbations}.  An admissible perturbation is a metric which differs from a reference
metric by a linear combination of symmetric products of the forms $\frac{dx}{x}$, $d_{y_i}$, and $xd_{z_j}$ for 
base coordinates $y_i$ and fibre coordinates $z_j$, where the coefficients vanish like $x^2$ as $x\to 0$.
If we follow the description of the structure of the Laplacian from Section 3 of \cite{GHAalto}, we see that such 
a perturbation in the metric results in perturbations of the resulting Laplacian by terms that vanish to the order
$x^{2aq}$, where $q$ is the order of the term in the perturbation.  Thus, by the definition of $\Pi$-split in section 5.1 of that paper, the Laplacian for the perturbed metric is still $\Pi$-split, thus the same Fredholm and regularity results hold as for the original product type and geometrically flat metric.

We also need a lemma about extensions of operators to weighted $L^2$ spaces.  Let $\delta_{fc,c}$ denote the formal adjoint of $d$ with respect to the natural metric on $x^cL^2\Omega^*(M, g_{fc})$.  Consider $d$, $\delta_{fc,c}$ and $\Delta_{fc,c}:=(d+\delta)_{fc,c}^2$ as unbounded operators on the Hilbert space $x^cL^2(M,g_{fc},\Lambda^*)$ with the pairing $<,>_{fc,c}$.  Recall that the domain of the maximal extension of an unbounded operator $P$ on $x^cL^2(M,g_{fc},\Lambda^*)$ with formal adjoint $P^*$ is defined by:

\begin{eqnarray*}
{\mathcal D}(P_{\max}) &=& \{\omega \in x^cL^2(M,g_{fc},\Lambda^*): P\omega \in x^cL^2(M,g_{fc},\Lambda^*)\} \\
& = & \{\omega \in x^cL^2(M,g_{fc},\Lambda^*): \exists \, \eta \in x^cL^2(M,g_{fc},\Lambda^{*})   \\
&\ &\quad \mbox{s.t.}\ \langle \omega, P^* \zeta\rangle = \langle \eta, 
\zeta\rangle \ \forall \zeta \in {\mathcal C}^\infty_0\Omega^*(M)\}.
\end{eqnarray*}
In other words, ${\mathcal D}(P_{\max})$ is the largest set of forms $\omega$ in $x^cL^2(M,g_{fc},\Lambda^*)$
such that $P\omega$, computed distributionally, is also in $x^cL^2$.

The minimal extension $P_{\min}$ is given by the graph closure of $P$ 
on ${\mathcal C}^\infty_0\Omega^*(M,\Lambda^*)$, i.e.
\begin{eqnarray*}
{\mathcal D}(d_{\min}) = \{\omega \in x^cL^2(M,g_{fc},\Lambda^*): \exists\, \omega_j \in
{\mathcal C}^\infty_0\Omega^*(M),\quad \omega_j \to \omega \ \mbox{in}\ x^cL^2\} \\
\mbox{and}\ P\omega_j \ \mbox{also converges to some}\ \eta \in
x^cL^2(M,g_{fc},\Lambda^*)\},
\end{eqnarray*}
in which case $P_{\min}\omega = \eta$. In particular, if the minimal and maximal extensions of $P$ coincide, then $P$ has a unique closed extension.  Further, if $P=P^*$ has a unique closed extension, then we say it is essentially self adjoint.  We then have the following.

\begin{lemma} 
The maximal and minimal closed extensions coincide for each of $d$, $\delta_{fc,c}$, $(d+\delta)_{fc,c}$ and $\Delta_{fc,c}:=(d+\delta)_{fc,c}^2$.  In particular, $(d+\delta)_{fc,c}$ and $\Delta_{fc,c}:=(d+\delta)_{fc,c}^2$ extend to self-adjoint operators
with respect to the pairing $<,>_{fc,c}$.
\end{lemma}
\begin{proof}
The essential self-adjointness of $\Delta_{fc,c}$ follows from Lemma 3.9 in \cite{BLhc} by considering the bundle 
$E$ to be the standard form bundle $\Omega^*(M)$ tensored with a trivial line bundle whose parallel sections under the 
standard flat connection have norm $x^{-c}$.  The other results then 
follow from this one by Lemma 3.8 in the same paper.
\end{proof}
We will use the same notation for the unique extensions of these operators as for the original unbounded operators.
Note that this lemma implies that if $d\eta \in x^cL^2$ and $\delta\tau \in x^cL^2$, then by definition, $\eta$ and $\tau$
are in the domains of $d$ and $\delta_{fc,c}$, respectively, and thus $\langle d\eta, \tau\rangle_{fc,c} = \langle \eta, \delta_{fc,c}\tau\rangle_{fc,c}$.
That is, there will be no boundary term entering in the integration by parts.  Similarly, if $(d+\delta)_{fc,c} \eta$
and $(d+\delta)_{fc,c} \tau$ are both in $x^cL^2$, then there is no boundary term when we integrate by parts
to shift the operator from one term to the other.
We can extend this a bit further by conjugating any of these operators by $x^a$.  Let $T$ be one of these
operators and assume $T\eta \in x^{c+a}L^2$ and $T^*\tau \in x^{c-a}L^2$.  Then $x^{-a}\eta$ is in the domain of 
$x^{-a}Tx^{a}$ extended to $x^cL^2$ and $x^a\tau$ is in the domain of its adjoint, $x^{a}T^*x^{-a}$.  Thus
\begin{equation}\label{extpairing}
\langle T\eta, \tau \rangle_{fc,c} = \langle x^{a}x^{-a}Tx^{a}x^{-a}\eta, \tau \rangle_{fc,c}
\end{equation}
\[
=\langle x^{-a}Tx^{a}x^{-a}\eta, x^{a}\tau \rangle_{fc,c}= \langle x^{-a}\eta, x^{a}T^*x^{-a}x^{a}\tau \rangle_{fc,c}
=\langle \eta,  T^*\tau \rangle_{fc,c}.
\]
We will refer to this as the extended $L^2$ pairing.


\section{Proof of Theorem \ref{l2hodge}}
For the first part of this theorem, we can let $g_{fc}$ be any fibred cusp metric at all.  This is because
fibred cusp metrics are all quasi-isometric, the space of $L^2$ harmonic forms is isomorphic to 
reduced $L^2$ cohomology, and reduced $L^2$ cohomology is a quasi-isometry invariant.  Thus
in this case, if we prove the theorem for product type fibred cusp metrics, we have up to isomorphism
proved it for any fibred cusp metric.
So assume that $g_{fc}$ is a product type fibred cusp metric.  By Lemma \ref{wcoho}, it suffices to show
\[\mathcal{H}_{L^2}^*(M,g_{fc},c) \cong \rm{Im}(WH^*(M,g_{fc},c+\epsilon) \to WH^*(M,g_{fc},c-\epsilon)).\]

In particular, we need to show that
\[\mathcal{H}_{L^2}^*(M,g_{fc},c) \cong \frac{{\rm Ker}(d) \subset x^{c+\epsilon} L^2\Omega^*(M,g_{fc})}{x^{c+\epsilon} L^2\Omega^*(M,g_{fc}) \cap d(x^{c-\epsilon}L^2\Omega^*(M,g_{fc}))}.\]

\paragraph{\bf The map:  }
Assume that $\omega \in \mathcal{H}_{L^2}^*(M,g_{fc},c)$.  Then by the regularity result,
$\omega \in x^{1+c+\epsilon}H^\infty_{split}$.
But then each $\omega_i$ is also in this space, which is contained in the domain of $(d+\delta_{fc,c})$
by the boundedness of split operators on split spaces.
Also, $(d+\delta_{fc,c})^2 \omega=0$, and because this operator preserves degrees, this is true for each 
$\omega_i$.  Thus $\omega_i$ is also in the domain of $(d+\delta_{fc,c})^2$.  But since $\omega_i$ is in the domain
of $d+\delta_{fc,c}$ and $(d+\delta_{fc,c})[(d+\delta_{fc,c})\omega_i)]=0$, this means also that $(d+\delta_{fc,c})\omega_i$
is in the domain of $(d+\delta_{fc,c})$.  Thus we are justified in integrating by parts by the self-adjointness
of $d+\delta_{fc,c}$, and we get
\[
||(d+\delta_{fc,c})\omega_j||^2_{x^cL^2} = \langle (d+\delta_{fc,c})\omega_j,(d+\delta_{fc,c})\omega_j\rangle_{x^cL^2}
=\langle (d+\delta_{fc,c})^2\omega_j,\omega_j \rangle_{x^cL^2} = 0.
\]
Thus also $(d+\delta_{fc,c})\omega_j=0$ for all $j=0, \ldots, n$.  But $d\omega$ is a $j+1$ form and $\delta_{fc,c}$ is a 
$j-1$ form, so they cannot cancel, which means that we have $d\omega_j = \delta_{fc_c} \omega_j=0$, so additionally,
$d\omega = \delta_{fc,c} \omega=0$ as required.

\paragraph{\bf Injectivity:  } 
Assume as before that $\omega \in \mathcal{H}_{L^2}^*(M,g_{fc},c)$ and that $\omega = d\eta$ for
$\eta \in x^{1+c-\epsilon}L^2\Omega^*(M,g_{fc})$.  Then from Part (1), $\delta_{fc,c} \omega=0$, so 
$\omega$ is in the domain of $\delta_{fc,c}$, and since $d\eta = \omega \in x^cL^2\Omega^*(M,g_{fc})$, also $\eta$ is in the domain
of $d$.  Thus we can write
\[
||\omega||^2_{x^c L^2} = <\omega, d\eta>_{x^c L^2} = <\delta_{fc,c}\omega, \eta>_{x^c L^2}=0,
\]
and get $\omega = 0$.  Thus the map is injective.

\paragraph{\bf Surjectivity:  } This comes from a Hodge decomposition theorem: 
\begin{lemma}\label{fchodge}
For any $\alpha \in x^c L^2_{fc}\Omega^*(M,g_{fc})$, there is a unique decomposition
\[
  \alpha = (d+\delta_{fc,c})\eta + \gamma,
  \]
where $\eta$ is in $x^{c+1-\epsilon} H^1_{split}\Omega^*(M,g_{fc})$ and
$\gamma\in \mathcal{H}_{L^2}(M,g_{fc},c)$.  

Further, if $\alpha \in x^{c+\epsilon}L^2_{fc}\Omega_{con}^*(M,g_{fc})$, then it can be uniquely
decomposed as
\[
  \alpha = d\eta_1 +\delta_{fc,c}\eta_2 + \gamma,
  \]
where each summand is independently in $x^{c+\epsilon}L^2_{fc}\Omega_{con}^*(M,g_{fc})$ and the terms
are mutually orthogonal as elements of $x^cL^2_{fc}\Omega^*(M,g_{fc})$.
\end{lemma}

\begin{proof}

We start with a weak Hodge decomposition theorem, which says that
for any $\alpha \in x^cL^2_{fc}\Omega^*(M,g_{fc})$, there is a unique decomposition
\[
  \alpha = \eta + \gamma,
\]
where $\eta$ is in the $x^cL^2$ closure of $(d+\delta_{fc,c})(x^{c+1}H^1_{split}\Omega^*(M,g_{fc}))$ and
$\gamma\in \mathcal{H}_{L^2}(M,g_{fc},c)$.  
To obtain this, we need to show that the orthogonal complement in $x^cL^2_{fc}\Omega^*(M,g_{fc})$
of the closure of $(d+\delta_{fc,c})(x^{c+1}H^1_{split}\Omega^*(M,g_{fc}))$ is 
$\mathcal{H}_{L^2}(M,g_{fc},c)$.  
First we can show the orthocomplement of the image is contained in harmonic forms.
Suppose that $\omega \in  x^cL^2_{fc}\Omega^*(M,g_{fc})$ has the property that
for all $\eta \in x^{c+1}H^1_{split}\Omega^*(M,g_{fc})$,
we have
\[
0 = \langle (d+\delta_{fc,c})\eta, \omega \rangle_{x^cL^2}.
\]
Take a sequence $\omega_n \to \omega$ in $x^c L^2_{fc}\Omega^*(M,g_{fc})$ where all of the $\omega_n$
are smooth and compactly supported.  Then we have 
\[
0 = \lim_{n \to \infty} \langle (d+\delta_{fc,c})\eta, \omega_n \rangle_{x^cL^2}
= \lim_{n \to \infty} \langle \eta, (d+\delta_{fc,c})\omega_n \rangle_{x^cL^2}.
\]
This implies $\lim_{n \to \infty} (d+\delta_{fc,c})\omega_n =0$, which exactly means that $\omega$ is in
the domain of $(d+\delta_{fc,c})$ and $(d+\delta_{fc,c})\omega=0$.  Thus
$\omega \in \mathcal{H}_{L^2}(M,g_{fc},c)$. Now, if $\omega \in \mathcal{H}_{L^2}(M,g_{fc},c)$, then 
we can directly integrate by parts to see that $\omega$ is orthogonal to the image
of $d+\delta_{fc,c}$, so the harmonic forms are also contained in the complement of the image.

 \medskip 
Next we identify the $x^cL^2$ closure of $(d + \delta_{fc,c})(x^{c+1}H^1_{split}\Omega^*(M,g_{fc}))$ as
\[
x^cL^2\Omega^*(M,g_{fc}) \cap (d+\delta_{fc,c})(x^{c+1-\epsilon} H^1_{split}\Omega^*(M,g_{fc})).
\]
To do this, we first let $\alpha \in x^{c+1-\epsilon} H^1_{split}\Omega^*(M,g_{fc})$, $(d+\delta)\alpha \in x^cL^2\Omega^*(M,g_{fc})$ and 
$\gamma\in \mathcal{H}_{L^2}(M,g_{fc},c)$.  Then 
\[
\langle (d+\delta_{fc,c})\alpha, \gamma \rangle_{x^cL^2} = \langle \alpha, (d+\delta_{fc,c})\gamma \rangle_{x^cL^2}=0
.\]
So $(d+\delta_{fc,c})\alpha$ is in the orthogonal complement of $\mathcal{H}_{L^2}(M,g_{fc},c)$.  But by the weak Hodge decomposition, this is equal
to the first space in the claim, so we get the first containment.

Now assume that $\alpha \in x^cL^2\Omega^*(M, g_{fc})$ is in the closure of the image of $d+\delta_{fc,c}$.  This means there is some sequence 
$b_n$ in $x^{c+1}H^1_{split}\Omega^*(M,g_{fc})$ with $|| \alpha - (d+\delta_{fc,c})b_n||_{x^cL^2} \to 0$
as $n\to \infty$.  But since $||\cdot||_{x^{c-\epsilon}L^2} \leq ||\cdot||_{x^cL^2}$, this means also that
$(d+\delta)b_n$ tends to $\alpha \in x^{c-\epsilon} L^2\Omega^*(M, g_{fc})$.  But ${\rm Im}(d+\delta_{fc,c})$ is closed here by the Fredholm
  property, so in fact $\alpha = (d+\delta_{fc,c})\eta$ for some $\eta \in x^{c+1-\epsilon}H^1_{split}\Omega^*(M, g_{fc})$, and we have the 
  second containment.   

\medskip
Finally, we can prove the second part of the lemma.  First assume we have such a three-term decomposition of $\alpha \in x^{c+\epsilon} L^2\Omega_{con}^*(M, g_{fc})$.  Then by integration by parts, the decomposition
is orthogonal with respect to $x^cL^2$; thus it is unique.  Furthermore, 
this means that if $d\alpha=0$, then $\delta_{fc,c} \eta_2=0$.  To show that such a decomposition exists, we only 
need that in the first equation, each of $d\eta$ and $\delta_{fc,c} \eta$ is independently in 
$x^{c+\epsilon} L^2\Omega^*(M, g_{fc})$.  (Here we are using the containment 
$x^{c+1-\epsilon}H^1_{split}\Omega^*(M, g_{fc}) \subset x^{c-\epsilon} L^2\Omega^*(M, g_{fc})$.)

Suppose first that $\alpha \in x^{c+\epsilon} L^2\Omega_{con}^*(M, g_{fc})$ and write $\alpha = d\eta + \gamma$ as in the decomposition
above.  We know $\gamma$ is also in $x^{c+\epsilon} L^2\Omega_{con}^*(M, g_{fc})$, so $(d+\delta_{fc,c})\eta \in x^{c+\epsilon} L^2\Omega_{con}^*(M, g_{fc})$.
From Theorem \ref{regfc} applied to $\eta$, this means that 
$\eta = \eta_0 + \eta'$, where $\eta' \in x^{c+1+\epsilon} H^1_{split}\Omega_{con}^*(M, g_{fc})$,
and on the end, 
\[
\eta_0 = u + \frac{dx}{x} \wedge v.
\]
By the same indicial root calculations as before, we get $u,v \in \calH^*(B,\calK^{(f/2)-c})$; that is, they are fibre harmonic forms on $Y$ of vertical degree $\frac{f}{2}-c$ that are also harmonic
as forms on $B$ with values in $\mathcal{K}_{(f/2)-c}$.  This means that $d\eta_0$ reduces to
\[
d \eta_0 = \left(\Pi x^2R\Pi  +\Pi^\perp d_Y \Pi \right) u + 
\frac{dx}{x}\wedge\left(\Pi x^2R\Pi  +\Pi^\perp d_Y \Pi \right) v, 
\]
where $\Pi$ is the operator that projects orthogonally onto fibre harmonic forms, $\Pi^\perp=I-\Pi$, and $R$ is a 0th order
operator corresponding to a curvature term.  
The regularity of the operator on the right here preserves the degree of vanishing of a conormal form in $x$ and 
$\eta \in x^{c+\epsilon} L^2\Omega_{con}^*(M, g_{fc})$, so also
$d \eta \in x^{c+\epsilon} L^2\Omega_{con}^*(M, g_{fc})$.  By an analogous argument, also 
$\delta_{fc,c} \eta\in x^{c+\epsilon} L^2\Omega_{con}^*(M, g_{fc})$.  

\end{proof}

As a corollary of the proof of this lemma, we can note the following alternative version of the $x^{c+\epsilon}L^2$ Hodge
decomposition:
\begin{corollary} 
\label{corNfc}
The following decomposition is orthogonal in $x^c L^2\Omega^*(M,g_{fc})$:
\[
x^{c+\epsilon}L^2_{fc}\Omega^*(M,g_{fc}) = (d+\delta_{fc,c})N_{fc,c} + \mathcal{H}_{L^2}(M,g_{fc},c),
\]
where 
\[
N_{fc,c} := \{ \eta \in x^{c+1-\epsilon}H^1_{split}\Omega^*(M,g_{fc}) : \eta |_{x<1/2} = u + \frac{dx}{x} \wedge v + \eta', \qquad \qquad \qquad \qquad\]
\[ \qquad \qquad \qquad
\eta' \in x^{c+1+\epsilon} H^1_{split}\Omega^*(M, g_{fc}); \,  u,v \in \calH^*(B,\calK^{(f/2)-c})\}.
\]
\end{corollary}

Now we can complete step (3) in the theorem.  We need to show that the map
\[
\mathcal{H}_{L^2}^*(M,g_{fc},c) \to \frac{{\rm Ker}(d) \subset x^{c+\epsilon} L^2\Omega^*(M,g_{fc})}{x^{c+\epsilon} L^2\Omega^*(M,g_{fc}) \cap d(x^{c-\epsilon}L^2\Omega^*(M,g_{fc}))}
\]
is a surjection.  Given the Hodge decomposition in Lemma \ref{fchodge}, this follows as in the standard Hodge theorem.
Let $\alpha \in x^{c+\epsilon} L^2\Omega^*(M,g_{fc})$ and $d\alpha=0$.  Then $\alpha = d\eta + \gamma$,
where $\gamma \in \mathcal{H}_{L^2}^*(M,g_{fc},c)$ and $\eta \in x^{c-\epsilon}L^2\Omega^*(M,g_{fc})$,
so $[\alpha] = [\gamma]$ in the cohomology group on the right.

\medskip
The last step is to consider the intersection product on these forms.
The $L^2$ intersection product on forms over $M$ with respect to the metric $g_{fc}$ gives a natural pairing
\[
\cap_{L^2}: x^{c}L^2\Omega^j(M,g_{fc}) \otimes x^{-c}L^2\Omega^{n-j}(M,g_{fc}) \to \mathbb{R}.
\]
\[
\alpha \cap_{L^2} \beta:= \int_M \alpha \wedge \beta.
\]
In particular, we can restrict to the space of weighted $L^2$ harmonic forms.  Because 
\[
*_c:\calH_{L^2}^*(M,g_{fc},c) \to \calH_{L^2}^*(M,g_{fc},-c)
\]
is an isomorphism, this means that for every nontrivial $\omega \in \calH_{L^2}^*(M,g_{fc},c)$, we can take
$*_c \omega \in \calH_{L^2}^*(M,g_{fc},-c)$ and get
\[
\omega \cap_{L^2} *_c\omega:= \int_M \omega \wedge *_c \omega := ||\omega||^2_c \neq 0.
\]
Thus it is nondegenerate on these spaces.  Furthermore, the isomorphism in Theorem \ref{l2hodge}
implies that the signature of this pairing in the case $c=0$ is exactly the perverse signature on $X$
with middle perversity, from which we obtain the desired result.


\section{Proof of Theorem \ref{exthodge}}

First we examine the space of extended $x^c L^2_{fc}$ harmonic forms.  We can see from Theorem \ref{regfc} that
these are the same as the $x^c L^2_{fc}$ harmonic forms when $\frac{f}{2}-c \notin \{0,\ldots,f\}$.  Thus we will 
restrict our consideration to $c \in \{-f/2, \cdots, f/2\}$. More importantly, we will also now restrict to the 
case where the metric on $M$ is geometrically flat.
From Corollary \ref{corNfc}, we get the following decomposition (note this does not require the flatness assumption).
\begin{lemma}
\label{lemhodge-ep}
There exists a vector space decomposition:
\[
x^{c-\epsilon}L^2\Omega^*(M, g_{fc}) = (d+\delta_{fc,c})(x^{c+1-\epsilon}H^1_{split}\Omega^*(M, g_{fc})) \oplus \mathcal{H}_{L^2}^*(M,g_{fc},c).
\]
\end{lemma}
\begin{proof}
We obtain this by dualising the decomposition in Corollary \ref{corNfc}.  That is, we show that $\mathcal{H}_{L^2}^*(M,g_{fb},c)$ is the 
annihilator under the extended $x^cL^2$ pairing of $(d+\delta_{fc,c})N_{fc,c}$ and $(d+\delta_{fc,c})(x^{c+1-\epsilon}H^1_{split}\Omega^*(M, g_{fc}))$ is the annihilator of $\mathcal{H}_{L^2}^*(M,g_{fb},c)$.  We see this as follows.

From Corollary \ref{corNfc} we have that $\mathcal{H}_{L^2}^*(M,g_{fb},c)$ is contained in the annihilator
of $(d+\delta_{fc,c})N_{fc,c}$.  Also, its dimension is equal to the codimension in $x^{c+\epsilon}L^2\Omega^*(M, g_{fc})$
of $(d+\delta_{fc,c})N_{fc,c}$, so in fact, it is the entire annihilator.  

Next, we have that $(d+\delta_{fc,c})(x^{c+1-\epsilon}H^1_{split}\Omega^*(M, g_{fc}))$ is contained in the annihilator of 
$\mathcal{H}_{L^2}^*(M,g_{fb},c)$ under the extended pairing.  We need to show that it is the whole annihilator.
Assume that $\alpha \in x^{c-\epsilon}L^2\Omega^*(M, g_{fc})$ and that $\langle \alpha, \gamma \rangle_{x^cL^2} = 0$
for all $\gamma \in \mathcal{H}_{L^2}^*(M,g_{fb},c)$.  
Since $x^{c+\epsilon}L^2\Omega^*(M, g_{fc})$ is dense in $x^{c-\epsilon}L^2\Omega^*(M, g_{fc})$,
we can take a sequence $\{\alpha_n\} \subset x^{c+\epsilon}L^2\Omega^*(M, g_{fc})$ such that
$\alpha_n \to \alpha$ in $x^{c-\epsilon}L^2\Omega^*(M, g_{fc})$.  Then we can decompose each $\alpha_n = (d+\delta_{fc,c})\beta_n + \tau_n$
using the decomposition from Corollary \ref{corNfc}.  
Let $\gamma_k$ be a basis of $\mathcal{H}_{L^2}^*(M,g_{fb},c)$.  We get
\[
0=\lim_{n \to \infty} \max_k \langle \alpha_n , \gamma_k \rangle_{x^cL^2} 
\]
\[= \lim_{n \to \infty} \max_k \langle (d+\delta_{fc,c})\beta_n + \tau_n , \gamma_k \rangle_{x^cL^2}
= \lim_{n \to \infty} \max_k \langle \tau_n , \gamma_k \rangle_{x^cL^2}.
\]
This implies that $\tau_n \to 0$ in $x^{c-\epsilon}L^2\Omega^*(M, g_{fc})$.
So in fact, we simply have $\alpha = \lim_{n \to \infty} (d+\delta_{fc,c})\beta_n$.  But the image of $(d+\delta_{fc,c})$ 
is closed in $x^{c-\epsilon}L^2\Omega^*(M, g_{fc})$, so in fact $\alpha \in (d+\delta_{fc,c})x^{c+1-\epsilon}H^1_{split}\Omega^*(M, g_{fc})$ as required.
\end{proof}

When the metric on $M$ is a geometrically flat fibred cusp metric, we get the following refined version of Lemma \ref{lemhodge-ep}.
\begin{lemma}\label{lemhodge-ep2}
If $g_{fc}$ is a geometrically flat fibred cusp metric, then there exists a vector space decomposition:
\[
x^{c-\epsilon}L^2\Omega^j(M, g_{fc}) = \Delta_{fc,c}(x^{c+2-\epsilon}H^2_{split}\Omega^j(M, g_{fc})) \oplus \mathcal{H}_{L^2}^j(M,g_{fc},c).
\]
\end{lemma}
\begin{proof}
Note that 
\[
x^{c+1-\epsilon}H^1_{split}\Omega^*(M,g_{fc}) \subset x^{c-\epsilon}L^2\Omega^*(M, g_{fc}),
\]
so we may apply the decomposition in Lemma \ref{lemhodge-ep} to this space.  We obtain
\[
x^{c+1-\epsilon}H^1_{split}\Omega^*(M,g_{fc}) = x^{c+1-\epsilon}H^1_{split}\Omega^*(M,g_{fc}) \cap 
\]
\[\left((d+\delta_{fc,c})(x^{c+1-\epsilon}H^1_{split}\Omega^*(M, g_{fc})) \oplus \mathcal{H}_{L^2}^*(M,g_{fb},c) \right).
\]
The second term on the right (second row) is already in $x^{c+1-\epsilon}H^1_{split}\Omega^*(M,g_{fc})$ by regularity.  
Also by regularity, if
$(d+\delta_{fc,c})\eta \in x^{c+1-\epsilon}H^1_{split}\Omega^*(M,g_{fc})$, 
then $\eta \in x^{c+2-\epsilon}H^2_{split}\Omega^*(M, g_{fc})$.  So this reduces to
\[
x^{c+1-\epsilon}H^1_{split}\Omega^*(M,g_{fc}) = (d+\delta_{fc,c})(x^{c+2-\epsilon}H^2_{split}\Omega^*(M, g_{fc})) \oplus \mathcal{H}_{L^2}^*(M,g_{fb},c).
\]
Substitute this into the decomposition from Lemma \ref{lemhodge-ep}.  The harmonic form piece from 
$x^{c+1-\epsilon}H^1_{split}\Omega^*(M,g_{fc})$ vanishes when we apply $(d+\delta_{fc,c})$.  Now we
simply note that the Laplacian preserves form degree so we may write decomposition in each degree
to get the desired result.
\end{proof}

Let us consider the extended harmonic forms.
\begin{lemma}
\label{extclosed}
Let $\eta \in \calH_{ext}^*(M,g_{fc},c)$.  Then $d\eta = \delta_{fc,c} \eta=0$ and $\eta_j \in \calH_{ext}^j(M,g_{fc},c)$ for all $j$.
\end{lemma}
\begin{proof}
First, by the regularity result, 
$\eta = \eta_0 + \eta'$, where $\eta' \in x^{c+1+\epsilon} H^1_{split}\Omega^*(M, g_{fc})$,
and on the end, 
\[
\eta_0 = x^{c-f/2}\left(x^{(f/2)-c}u_{(f/2)-c} + \frac{dx}{x} \wedge x^{(f/2)-c}v_{(f/2)-c} \right), 
\]
where $u,v \in \calH^*(B,\calK^{(f/2)-c})$.  As in the proof of Lemma \ref{fchodge} above, this means that
$d\eta  \in x^{c+\epsilon}L^2\Omega^*(M, g_{fc})$.  This in turn implies 
$\delta_{fc,c}\eta \in x^{c+\epsilon}L^2\Omega^*(M, g_{fc})$.
Thus $\eta$ is in the domain of $d$ and 
$\delta_{fc,c} \eta$ is in the domain of $\delta_{fc,c}$, so we get
\[
||d \eta||^2_{x^cL^2} = \langle d\eta, -\delta_{fc,c} \eta \rangle_{x^cL^2} =  \langle \eta, -\delta_{fc,c}^2 \eta \rangle_{x^cL^2}=0.
\]
So $d\eta = 0$, and thus also $\delta_{fc,c} \eta=0$.

We now can notice that when we decompose by degree, $d \eta_j=0$ for all $j$, as $d$ shifts all of the form
degrees up by one.  Similarly, $\delta_{fc,c} \eta_j=0$.  Putting these together shows that $(d+\delta_{fc,c})\eta_j=0$
for all $j$.
\end{proof}

When the metric on $M$ is geometrically flat, we can apply the decomposition in Lemma \ref{lemhodge-ep2} to the 
extended harmonic forms to get the following.
\begin{lemma} \label{4.4} If $g$ is a geometrically flat fibred cusp metric on $M$ then
\[
\calH_{ext}^j(M,g_{fc},c) = d(null(\Delta_{fc,c}^{j-1})) + \delta_{fc,c}(null(\Delta_{fc,c}^{j+1})) + \calH_{L^2}^j(M,g_{fc},c),
\]
where $null(\Delta_{fc,c}^j):=\{ \eta \in x^{c+1-\epsilon}H^1_{split}\Omega^*(M, g_{fc}) \mid \Delta_{fc,c}\eta=0\}$
and where the sums are vector space direct sums.
\end{lemma}

\begin{proof}
By Lemma \ref{lemhodge-ep2}, we have 
\[
\calH_{ext}^j(M,g_{fc},c) = \Delta_{fc,c}W^j + \calH_{L^2}^j(M,g_{fc},c),
\]
where $\Delta_{fc,c}W^j = \calH^j_{ext}(M,g_{fc},c) \cap \Delta_{fc,c}(x^{c+2-\epsilon}H^2_{split}\Omega^j(M, g_{fc}))$.
Note by bounded mapping properties, both $d\delta_{fc,c}W^j$ and $\delta_{fc,c}W^j$ are in 
$x^{c-\epsilon}L^2\Omega^j(M, g_{fc})$.  Furthermore, applying $d$ to both sides tells us by Lemma
\ref{extclosed} that $d\delta_{fc,c}dW^j=0$.  Similarly, applying $\delta_{fc,c}$ to both sides we get that
$\delta_{fc,c}d\delta_{fc,c}W^j=0$.  Thus $d\delta_{fc,c}W^j$ and $\delta_{fc,c}dW^j$ are independently
in $\calH_{ext}^j(M,g_{fc},c)$.  Thus we may write
\begin{equation}\label{Wj}
\calH_{ext}^j(M,g_{fc},c) = d\delta_{fc,c}W^{j} + \delta_{fc,c}dW^{j} + \calH_{L^2}^j(M,g_{fc},c).
\end{equation}

Now let us consider what $W^j$ looks like.  We can apply $d$ to both sides of this equation to see that
$d\delta_{fc,c}dW^j=0$.  Similarly, applying $\delta_{fc,c}$ to both sides, we get that $\delta_{fc,c}d\delta_{fc,c}W^j=0$.
This means that $dW^j$ and $\delta_{fc,c}W^j$ are both contained in $null(\Delta_{fc,c}^*)$.  In particular, we have that
$dW^j \subset null(\Delta_{fc,c}^{j+1})$ and $\delta_{fc,c}W^j \subset null(\Delta_{fc,c}^{j-1})$.  

Now consider $\tau \in null(\Delta_{fc,c}^{j})\subset x^{c-\epsilon}L^2\Omega(M,g_{fc})$.  
Then $(d+\delta_{fc,c})\tau \in \calH_{ext}^j(M,g_{fc},c)$, but by Lemma \ref{extclosed} again, all such forms are 
both closed and coclosed, so $d\delta_{fc,c}\tau=\delta_{fc,c}d\tau=0$ independently.  This means
that also $d\tau, \delta_{fc,c}\tau \in \calH_{ext}^j(M,g_{fc},c)$ independently, so we may
replace $\delta_{fc,c} W^j$ by $null(\Delta_{fc,c}^{j-1})$ and $d W^j$ by $null(\Delta_{fc,c}^{j+1})$
in Equation (\ref{Wj}).  

Finally, we already know that the third term has trivial intersection with the sum of the first two.
But  if $\tau \in null(\Delta_{fc,c}^{*})$, then by the Theorem \ref{regfclap}, 
$\tau = \tau_0 + \tau'$, where $\tau' \in x^{c+1+\epsilon} H^1_{split}\Omega_{con}^*(M, g_{fc})$, and 
on the end, 
\begin{equation}\label{harmend}
\eta_0 = \left(u_{11}\log x + u_{21}+ \frac{dx}{x} \wedge (u_{12}\log x + u_{22}) \right), 
\end{equation}
where $u_{i,j} \in \calH^*(Y)$.  We have that $\delta_{fc,c}=\delta_{fc} - 2c\frac{dx}{x} \lrcorner$.
This means that $\delta_{fc}\eta_0$ will have no $\frac{dx}{x}$ part, so $*_c \delta_{fc}\eta_0$ will
have only the $\frac{dx}{x}$ part.  Now consider the boundary term in
the integration by parts: 
\[
<d\tau, \delta_{fc,c} \eta>_{x^cL^2} = \int_M d\tau \wedge *_c \delta_{fc,c}\eta=\lim_{x\to 0} \int_Y \tau(x) \wedge *_c \delta_{fc,c}\eta(x)=0.
\]
Because by hypothesis $\tau$ and $\eta$ are polyhomogeneous, as they are in the kernel of 
the Laplacian, the boundary terms containing $\tau'$ and $\eta'$ will vanish to some positive
degree in $x$ at the boundary.  
However, because the remaining term has a $\frac{dx}{x}$ in it, this term is identically
zero when integrated over $Y$.  Thus we can integrate by parts to show that the first
two terms in the decomposition also have trivial intersection, which finishes the proof.
\end{proof}

This leads us to make the following definitions.  Let $p =(f/2)-c \in \{0,\ldots,f+1\}$.  Define
\[
\mathcal{IH}^j_{(f/2)-c-\epsilon}(M,g_{fc}) := d(null(\Delta_{fc,c}^{j-1})) + \calH_{L^2}^j(M,g_{fc},c),
\]
and
\[
\mathcal{IH}^j_{(f/2)-c+\epsilon}(M,g_{fc}) := \delta_{fc,c}(null(\Delta_{fc,c}^{j+1})) + \calH_{L^2}^j(M,g_{fc},c).
\]
Now we can prove the version of Poincar\'e duality on extended harmonic forms.
\begin{lemma}\label{PD}
The weighted Hodge star operator, $*_c$, gives
an isomorphism
\begin{equation}\label{poincare}
*_c: \mathcal{IH}^j_{(f/2)-c-\epsilon}(M,g_{fc}) \to \mathcal{IH}^{n-j}_{(f/2)+c+\epsilon}(M,g_{fc}).
\end{equation}
\end{lemma}
We have that $\omega \in \mathcal{H}^j_{ext}(M,g_{fc},c)$ if and only if
$\omega \in x^{c-\epsilon}L^2\Omega^*(M,g_{fc})$ and $d\omega = \delta_c \omega =0$.
So suppose we have $\omega \in \mathcal{H}^j_{ext}(M,g_{fc},c)$.  Then since $*_c = x^{-2c}*$, 
we get $*_c\omega \in x^{-c-\epsilon}L^2\Omega^*(M,g_{fc})$.  Next we can use the
equations $*_c*_{-c}=(-1)^{j(n-j)}$ and $\delta_{fc,c} = (-1)^{n(j+1)+1}*_{-c}d*_c$ to write 
\[
\delta_{fc,-c} *_c \omega  = (-1)^{n(j+1)+1}*_c d *_{-c}*_c \omega =(-1)^{n-j+1} *_c d\omega = *_c 0 = 0,
\]
and
\[
d (*_c \omega) = (-1)^{(n-j+1)(j-1)}*_c (*_{-c}d*_{c}) \omega = (-1)^{j^2}*_c \delta_{fc,c} \omega = (-1)^{j^2}*_c 0 = 0.
\]
Thus $*_c \omega \in \calH_{ext}^{n-j}(M,g_{fc},-c)$.
Finally, we can note that if $\omega \in \mathcal{IH}^j_{(f/2)-c-\epsilon}(M,g_{fc})$, then
\[
*_c \omega = *_c(d\eta + \gamma) = \delta_{-c} ((-1)^{(j-1)(n-j+1)+nj+1}*_c \eta) + *_c \gamma.
\]
By the same argument that shows $*_c \omega \in \calH_{ext}^{n-j}(M,g_{fc},-c)$, we get that 
$*_c \gamma \in \calH_{L^2}^{n-j}(M,g_{fc},-c)$ and that 
\[(-1)^{(j-1)(n-j+1)+nj+1}*_c \eta \in x^{-c-\epsilon}L^2\Omega^*(M,g_{fc})\]
and 
\[
\Delta_{fc,-c} ((-1)^{(j-1)(n-j+1)+nj+1}*_c \eta)=0.
\]
So $*_c \omega \in \mathcal{IH}^j_{(f/2)+c+\epsilon}(M,g_{fc})$, as desired.

\medskip
The remainder of Theorem \ref{exthodge} is then a corollary of the following two lemmas.
\begin{lemma}\label{GI2}
Let $g$ be a geometrically flat fibred cusp metric on $M$.  Then 
\[
\mathcal{IH}^j_{(f/2)-c+\epsilon}(M,g_{fc}) \cong IH^j_{(f/2)-c+\epsilon}(X,B).
\]
\end{lemma}

\begin{proof}[Proof of Lemma \ref{GI2}]
We make use of the isomorphism in Lemma \ref{wcoho}
to rewrite 
\[
IH^j_{(f/2)-c+\epsilon}(X,B) \cong WH^j_{fc}(M,c-(\epsilon/2)).
\]
So we need to show that
\begin{equation}
\label{uppermid}
\delta_{fc,c}(null(\Delta_{fc,c}^{j+1})) + \calH_{L^2}^j(M,g_{fc},c) \cong WH^j_{fc}(M,c-\epsilon).
\end{equation}

We already know that there is a map from the right side to the left that simply takes 
an element $\omega$ on the right to its class in $WH^j_{fc}(M,c-\epsilon)$.  This makes sense
because we know the right side is contained in $\calH_{ext}^j(M,g_{fc},c)$, thus  
is in $x^{c-\epsilon}L^2\Omega^j(M,g_{fc})$ for all $\epsilon>0$.  Further, elements on the left are closed
by Lemma \ref{extclosed}.  Thus we need to show that this map is both injective and surjective.

First look at surjectivity.  Suppose that $\alpha \in x^{c-(\epsilon/2)}L^2\Omega^j(M,g_{fc})$ and that $d\alpha=0$.
Then by Lemma \ref{lemhodge-ep2}, we can decompose it as
\[
\alpha = \Delta_{fc,c}(\tau) + \gamma,
\]
where $\tau \in x^{c+1-\epsilon}H^1_{split}\Omega^j(M,g_{fc})$ and $\gamma \in \calH^j_{L^2}(M,g_{fc},c).$
Now applying $d$ to both sides, we get that $d\delta_{fc,c} d\tau=0$.  But this implies that in fact 
$\Delta_{fc,c}(d\tau)=0$, that is $d\tau \in null(\Delta_{fc,c}^{j-1})$.  
Let $\omega = \delta_{fc,c} d\tau + \gamma$, which is in the space on the left in Equation \ref{uppermid}.
Then we have $\alpha = \omega + d(\delta_{fc,c} \tau)$, where $\delta_{fc,c}\tau \in x^{c-(\epsilon/2)}L^2\Omega^j(M,g_{fc})$.
Thus the map is surjective.

Now consider injectivity.  Suppose $\omega \in \delta_{fc,c}(null(\Delta_{fc,c}^{j+1})) + \calH_{L^2}^j(M,g_{fc},c)$.
Then from Lemma \ref{extclosed} we also know that $\delta_{fc,c} \omega=0 \in x^{c+\epsilon}L^2\Omega^j(M,g_{fc})$.
Assume $\omega = d\alpha$ for some $\alpha \in x^{c-\epsilon}L^2\Omega^j(M,g_{fc})$.  Then we can 
use the extended pairing to get
\[
||\omega||^2_c = \langle d\alpha,\omega \rangle_{c} = \langle \alpha, \delta_{fc,c} \rangle_c=0.
\]
Thus we get that $\omega=0$, and the map is also injective.
\end{proof}

\begin{lemma}\label{GI1}
Let $g$ be a geometrically flat fibred cusp metric on $M$.  Then 
\[
\mathcal{IH}^j_{(f/2)-c-\epsilon}(M,g_{fc}) \cong IH^j_{(f/2)-c-\epsilon}(X,B) .
\]
\end{lemma}

\begin{proof}[Proof of Lemma \ref{GI1}]
We can do this formally as follows.  By Lemma \ref{PD}, 
\[
\mathcal{IH}^j_{(f/2)-c-\epsilon}(M,g_{fc}) \cong \mathcal{IH}^{n-j}_{(f/2)+c+\epsilon}(M,g_{fc}).
\]
By Lemma \ref{GI2}, the space on the right is isomorphic to $IH^{n-j}_{(f/2)+c+\epsilon}(X,B)$, 
which by Poincar\'e duality for intersection cohomology is isomorphic to 
$IH^{j}_{(f/2)-c-\epsilon}(X,B)$, as desired. 

\end{proof}

\section{Proof of Theorem \ref{les}}
In order to prove Theorem \ref{les}, we need to show that in the geometrically flat setting, the $f/2-c-\epsilon$ and $f/2-c+\epsilon$
perversity intersection cohomologies fit into the indicated long exact sequence.  Additionally, to realise the 
map from the spaces of extended harmonic forms, we also need to show that we can calculate the $f/2-c-\epsilon$ perversity intersection cohomology from a different complex of forms.  

\subsection{Cochain complexes for relative and absolute cohomology}

Before we expand the possible complexes for calculating intersection cohomologies of various perversities, it is instructive to collect in one place various complexes that can calculate the corresponding cohomologies in the case when $f=c=0$, in which case $X$ is a manifold with boundary $B$ and $IH^*_{\frac{f}{2}-c}(X,B) \cong H^*(X-B)$
and $IH^*_{\frac{f}{2}-c-1}(X,B) \cong H^*(X,B)$.  Although some of the complexes used to calculate these cohomology groups are considered in other papers, some
are new.  In addition, there is not any place so far in the literature where all of the various complexes are laid out together for consideration.  Thus it is independently useful to record these here.

To avoid confusion later, we will refer here to the manifold
with boundary $X=\olM$.  In this case, we know that $H^*(\olM) \cong H^*(\olM - \partial \olM) = H^*(X-B)$.  We will start
by considering $H^*(\olM)$, which may be calculated from the complex of smooth forms on $\olM$.  Consider the following 
exact sequence of complexes:
\begin{equation}
\label{cx1}
0 \to \Omega_{0}(\olM) \stackrel{inc}{\longrightarrow} \Omega(\olM) \stackrel{i^*_{\partial \olM}}{\longrightarrow} \Omega(\partial \olM) \to 0,
\end{equation}
where $i^*_{\partial\olM}$ is pullback to the boundary induced from the inclusion $i_{\partial\olM}:\partial\olM \to \olM$
and $\Omega_0(\olM)$ is the kernel of this map.  This induces a long exact sequence on cohomology, which by the 
five lemma is the relative cohomology exact sequence.  Thus the complex 
\[
\Omega_0(\olM) := \{\omega \in \Omega(\olM) \mid i^*_{\partial\olM} \omega =0\}
\]
calculates relative cohomology.

Now it is well known that $H^*(\olM) \cong H^*(M)$, where the right side is calculated from the much larger complex
of smooth forms on the open manifold $M$.  The isomorphism of these two complexes comes from identifying
$\olM$ with a deformation retract of $M$ to $M_s:= \olM - \left([0,s) \times Y\right)$.  When we chase through the isomorphisms,
we get that the maps in the relative cohomology long exact sequence are then induced from the short exact sequence
of complexes:
\begin{equation}
\label{cx2}
0 \to \Omega_s(M) \stackrel{inc}{\longrightarrow} \Omega(M) \stackrel{i^*_{s}}{\longrightarrow} \Omega(Y) \to 0,
\end{equation}
where $i^*_s$ is pullback to the submanifold $Y \times \{s\}$ induced from the inclusion $i_{s}:Y\cong (Y \times \{s\}) \to M$
and $\Omega_s(\olM)$ is the kernel of this map.  Thus relative cohomology may be calculated from any
of these complexes.  We can push this even further to get the following new lemma.

\begin{lemma}\label{rel0cx}
Consider the subcomplex
\[
\Omega_0^*(M):= \{ u \in \Omega^*(M) \mid \lim_{s \to 0} i^*_s u =0 \},
\]
where we fix an identification of the collar neighborhood $Y\times (0,1)$ and a metric on $Y$ and define the limit using the 
$L^2$ norm on forms over $Y$.  Call its cohomology
$H_0^*(M)$.  Then $H_0^*(M) \cong H(\olM,\partial\olM)$.
\end{lemma}
\begin{proof}
We prove this again using the relative cohomology long exact sequence and the five lemma.
First we need to show exactness at $H^*(M)$ in the long exact sequence.
The complex $\Omega_0^*(M)$ is a sub complex of $\Omega^*(M)$, thus there is an induced inclusion map on cohomology.
We want to show that the image of this map is equal to the kernel of $i^*_s$.  
The homotopy used for the proof of $H^*(\olM) \cong H^*(M)$ gives that
if $[\omega] \in H^*(M)$ and $i_{s}[\omega]=[0] \in H^*(Y)$ for some $s$, then in fact $i_{s}[\omega]$ vanishes for all $s \in (0,1]$.
This is because if $F_s$ is a deformation retract of $M$ onto $M_s$ for $s \in (0,1]$, then $F^*_s u$
is a form on $\olM$ given by restriction of $M$ to $M_s$ and identification of $M_s$ with $\olM$.
The chain homotopy is given by:
\[
R_{t,s}(u):= \int_t^s F_x^*(dx \lrcorner u)\, dx,
\]
and $F_s^*(u) - F_t^*(u) = dR_{t,s}u + R_{t,s}(du)$ and $i^*_s = i^*_{\partial\olM} F^*_s$.
So if $u$ is closed and $i^*_t u =0$ for some $t$, then
\[
i_s^* u = i_s^* u - i^*_t u = i^*_{\partial\olM}(F_s^*(u) - F_t^*(u)) =  i^*_{\partial\olM} (dR_{t,s}u) = d_{Y}i^*_{\partial\olM} (R_{t,s}u).
\]
Thus $i^*_s u$ is exact for all $s \in(0,1]$, and $0= i^*_s[u] \in H^*(Y)$ for all $s$.  We can thus write $u$ on the end as
\[
u= d_Y \alpha + \frac{dx}{x} \beta
\]
for $\alpha,\beta$ smooth families of forms on $Y$.
Now let $\chi$ be a cutoff function supported near $\partial \olM$.
Then $i_s^*(u)(u - d\chi \alpha)=0$ for all $s$, so in particular, $\lim_{s\to 0} i_s^*(u)=0$,
and $[u] = [u - d\chi \alpha]$.  This means that the kernel of $i^*_s$ is contained in the image of $H_0^*(M)$.

\smallskip
We next need to show that the image of $H_0^*(M) \to H^*(M)$ is contained in the kernel of $i^*_s$.  To do this, we define a metric on $H^*(\partial M)$.
This comes from any fixed metric on $\partial M$, using the 
$L^2$ norm on the space of harmonic forms, which is the same as the metric given on cohomology by taking the minimum
$L^2$ metric for any form representing a given class.
Then since the restriction map is constant in $s$ on cohomology, and the $L^2$ norm of the harmonic
projection is always bounded by the $L^2$ norm of the original form, we get that if $u \in \Omega^*(M)$
is closed and $\lim_{s \to 0} i^*_s u =0$, then actually $i^*_s[u]=0$ for all $s$.  Thus it is in the kernel
of the map $i^*_s$ as required.

\smallskip

In order to check exactness at $H_0^*(M)$, we first need to consider the boundary map.  The $s$-complex boundary map takes a closed form 
$\alpha \in \Omega^{j-1}(\partial \olM)$ and extends it to any form $\tilde{\alpha}$ where $i^*_s\tilde{\alpha} =\alpha$.
Then $\delta[\alpha]_{\partial \olM} = [d\tilde{\alpha}] \in H_s^*(M)$.  Now
we need it to extend $\alpha$ to a form $d\tilde{\alpha}$ where $\lim_{s \to 0} i^*_s \tilde{\alpha}=0$.  We can do 
this by taking the form $\tilde{\alpha}:= \chi \pi^*\alpha$ where $\pi:U \to \partial M$.  Since $\tilde\alpha$ is 
actually in $\Omega^{j-1}(M)$, we get that $i^*_s (d\chi \pi^*\alpha) = d (i^*_s(\chi \pi^*\alpha)) = d\alpha =0$ for all $s\in[0,1/2]$,
and thus also the limit works.

Now we can check the exactness at  $H_0^*(M)$. If $u \in \Omega^j_0(M)$
and $u = dw$ for $w \in \Omega^{j-1}(M)$, then as in the usual exact sequence, $[u] = \delta[i^*_s w]$, thus the kernel
of the map to $H^j(M-\partial M)$ is contained in the image of $\delta$.  Finally, suppose that $[u] = \delta[\alpha]$.
By definition of $\delta$, this means $u = d\tilde{\alpha}$ for some $\alpha \in \Omega^{j-1}(M)$.  This means that 
the image of $\delta$ is in the kernel of the map induced by inclusion of complexes.
\end{proof}

Finally, we would like to combine what we know about relative and absolute cohomology for the manifold $M$
using smooth forms and what we know about them using forms which are in $L^2$ with respect to a b-metric,
which for the purposes of this paper we may take to be a metric that on the end of $M$ has the form
\[
g_b = \frac{dx^2}{x^2} + g_Y.
\]
We can start with the complex of conormal $x^cL^2$ forms over $M$ with respect to $g_b$, which is 
defined by:
\[
x^cL^2_{con}\Omega^*(M,g_b) := \{ \omega \in \Omega^*(M) \mid \, ||x^{-c} V_1 \cdots V_j \omega||_{L^2_g}< \infty \,
\]
\[
\hspace{1in} \forall \, V_i \mbox{ bounded vector fields with respect to the metric } g_b\}.
\]
Because we have near the boundary that $d_M = \frac{dx}{x} \wedge x\partial_x + d_Y$, where $\frac{dx}{x}$ is 
a pointwise unit length 1-form and $x\partial_x$ is a unit length vector field, the conormality condition 
ensures that this is a complex.  If $\epsilon>0$, then the cohomology of this complex is 
\[
L^2H_{con}^*(M,g_b,-\epsilon) \cong H^*(M),
\]
as is proved, for example, in \cite{Me-aps}.  The complex $x^{-\epsilon}L^2_{con}\Omega^*(M,g_b)$ is a subcomplex of $\Omega^*(M)$, 
and the same map $i^*_s$ can be defined to $H^*(Y)$.  We can use this to prove a new $L^2$ version of Lemma
\ref{rel0cx}.
\begin{lemma}\label{conrel0cx}
Define the subcomplex
\[
x^{-\epsilon}L^2_{0}\Omega^*(M,g_b):= \{ \omega \in x^{-\epsilon}L^2\Omega_{con}^*(M,g_b) \mid \,\lim_{s\to 0} i^*_s(\omega)=0 \}.
\]
Then its cohomology is isomorphic to relative cohomology, i.e., $L^2H^*_0(M,g_b,-\epsilon) \cong H^*_0(M)$.
\end{lemma}

\begin{proof}
Again, we use the five lemma and the relative long exact sequence.  All we need to do is check in the proof
of Lemma \ref{rel0cx} that things descend to these subspaces.  First we can check the boundary map.
We have for $\alpha \in H^*(Y)$ that $\delta[\alpha] = [d\chi \pi^* \alpha]$.  This form has constant pointwise
norm in the b-metric on the end, so it is in $x^{-\epsilon} L^2$.  Also, since $\alpha$ and $\chi$ are smooth,
it is conormal.  Finally, it satisfies the vanishing condition for $i^*_s$.  Thus the image of this map lands
in $L^2H^*_0(M,g_b,-\epsilon)$.  The proof of exactness of the long exact sequence at $L^2H^*_0(M,g_b,-\epsilon)$
now follows as before.

When we check the exactness at $L^2H_{con}^*(M,g_b,-\epsilon)$, we just need to check that
the form $\alpha$ constructed via the homotopy is in the conormal complex.  We basically
have that 
\[
\alpha(s) = \int_t^s i^*_xu \, dx,
\]
where by conormality, we know that $i^*_x u = \mathcal{O}(x^{-\epsilon})$.  When we integrate, we will get something
$ \mathcal{O}(x^{1-\epsilon})$, which is again conormal and in the correct weighted space. 
So the remainder of the argument follows as before.
\end{proof}


\subsection{Cochain complexes for intersection cohomology}
We would like to define complexes that calculate intersection cohomology analogous to the complexes 
for relative and absolute cohomology defined in the previous subsection.

For the proof of Theorem \ref{exthodge}, we need a complex analogous to $x^cL^2_{0}\Omega^*(M,g_b)$ for relative
cohomology.  We can't define this in general--it will require that 
the link bundle $Y \to B$ is geometrically flat.  Recall that geometrically flat means that there is some fixed
metric $g_F$ on $F$ so that $Y$ is flat with respect to $\mbox{Isom}(F,g_F)$.  We may then endow $Y$ with the metric
such that for any local trivialisation $V \times F \to V$, the metric is a product of $g_F$ on the $F$ factor 
and a metric on $V \subset B$ restricted from some (again fixed) choice of metric $g_B$ on $B$.
We have a Hodge theorem on $F$ with respect to the fixed metric $g_F$:
\[
\Omega^j(F) = d\Omega^{j-1}(F) \oplus \delta\Omega^{j+1}(F) \oplus \mathcal{H}^j(F),
\]
where $\mathcal{H}^j(F)$ is the space of harmonic $j$ forms on $F$ and the $\oplus$ are orthogonal
sums with respect to the $L^2$ inner product on $\Omega^j(F)$ induced from the metric $g_F$.

The fact that the metric $g_{fc}$ is geometrically flat means that the Hodge decomposition
is preserved by the transition functions for our charts on the end of $M$, so it makes sense to talk
about the corresponding decomposition of forms on the end.  Further,
the vector fields $x\partial_x$ and 
any vector field lifted from the base preserve the Hodge decomposition on the fibres.

Define the complex
\[
\Omega_{c\tau, p}^i(F)= \left\{
\begin{array}{ll}
\{0\} & j \leq p-1 \\
\mbox{ker}(\delta_F) & j=p \\
\Omega^j(F) & j > p.
\end{array} \right.
\]
Now define 
\begin{equation}
\Pi_p:\Omega^*(F) \to \Omega_{c\tau, p}^*(F),
\end{equation}
where for $\omega = d_F\alpha + \delta_F \beta + \gamma \in \Omega^j(F)$,
\[
\Pi_p (\omega) = \left\{
\begin{array}{ll}
0& j\leq p-1 \\
\delta_F \beta + \gamma & j= p\\
\omega & j>p.
\end{array}
\right.
\]
Note that $d_F\Pi_p = \Pi_{p} d$, since if $j< p-1$, the $\Pi_p$ on both sides are just 0.  If $j=p-1$, then on the left, $\Pi_p=0$ 
and on the right, we use the fact that uniqueness of the Hodge decomposition in degree $p$ means $\Pi_p d_F=0$.
If $j=p$, then for $\omega = d_F\alpha + \delta_F \beta + \gamma$, we get on the right $d_F(\delta_F \beta +\gamma)= d_F(\omega) = \Pi_p (d_F\omega)$.
Finally, if $j > p$, then $\Pi_p$ on both sides is just the identity, so it commutes with $d_F$.

We can write the following short exact sequence:
\[
0\to \Omega^*_p(F) \stackrel{inc}{\longrightarrow} \Omega^*(F) \stackrel{\Pi_p}{\longrightarrow} \Omega_{c\tau, p}^i(F) \to 0,
\]
where $\Omega^*_p(F)$ is the kernel of $\Pi_p$.  This means that
\[
\Omega^j_p(F)=\left\{
\begin{array}{ll}
\Omega^j(F) & j \leq p-1 \\
d\Omega^{p-1}(F) & j=p \\
0 & j > p.
\end{array} \right.
\]
Note that the cohomology of this complex is exactly the local calculation for $IH_p(C(F))$.

Now we extend these constructions and this lemma to $Y$.  To do this, 
we write 
\[
\Omega^*(Y) \cong \Omega^*(B,\Omega^*(F)) \cong \sum_{j=0}^f  \Omega^*(B,\Omega^j(F)).
\]
Now we can use the Hodge decomposition with respect to the metric $g_F$ on each fibre to define the subcomplex
\[
\Omega_{c\tau,p}^*(Y) := \Omega^*(B,\Omega_{c\tau,p}^*(F)).
\]
This is a subcomplex because in each trivialisation, the exterior derivative $d_Y= d_B + d_F$, and $d_B$ commutes
with the Hodge decomposition on $\Omega^*(F)$.  Similarly we get a projection map:
\[
\Pi_p:\Omega^*(Y) \to \Omega_{c\tau,p}^*(Y).
\]
This again commutes with $d_Y$, as well as with any derivatives in the base directions.

Next, we can define complexes of forms on $\olM$.  Consider the complex of smooth forms on $\olM$ that fits into the following 
short exact sequence:
\[
0 \to I_p\Omega^*_0(\olM) \stackrel{inc}{\longrightarrow} \Omega^*(\olM) \stackrel{\Pi_p\circ i^*_{\partial \olM} }{\longrightarrow} \Omega_{c\tau,p}^*(Y) \to 0.
\]
Because derivatives in the base directions commute with $\Pi_p$, the complex generates a complex of fine sheaves on $X$.
Then the Kunneth formula applies over product-type regions $U \times [0,1] \times F$ to show that the local calculation
is the correct one for $IH^*_p(X)$.  Thus the cohomology of $I_p\Omega^*_0(\olM)$ is $IH^*_p(X)$.

We can also create a short exact sequence that relates intersection forms of two different perversities.  Here we will consider
only adjacent perversities, but this may be generalised.  Define
\[
\Omega_{c\tau, p-1,p}^i(F)= \left\{
\begin{array}{ll}
0 & j \leq p-1 \\
\mbox{ker}(\delta_F) & j=p \\
d\Omega^p(F) & j =p+1\\
0 & j> p+1.
\end{array} \right.
\]
As before, extend this to $Y$ to get $\Omega_{c\tau, p-1,p}^i(Y)$.

Now define 
\begin{equation}
\Pi_{p-1,p}:\Omega^*(F) \to \Omega_{c\tau, p-1,p}^*(F),
\end{equation}
where for $\omega = d_F\alpha + \delta_F \beta + \gamma \in \Omega^j(F)$,
\[
\Pi_{p-1,p} (\omega) = \left\{
\begin{array}{ll}
0& j\leq p-1 \\
\delta_F \beta + \gamma & j= p\\
d_F\alpha & j=p+1\\
0& j>p+1.
\end{array}
\right.
\]
As before, extend this to a map from $\Omega^*(Y)$ to $\Omega_{c\tau, p-1,p}^i(Y)$.
Now we obtain the short exact sequence:
\[
0 \to I_{p-1}\Omega^*_0(\olM) \stackrel{inc}{\longrightarrow} I_p\Omega^*_0(\olM) \stackrel{\Pi_{p-1,p}\circ i^*_{\partial \olM} }{\longrightarrow} \Omega_{c\tau,p-1,p}^*(Y) \to 0.
\]
The point of this is that we now get a long exact sequence relating $IH_{p-1}^*(X)$, $IH_p^*(X)$ and $H^*(B,H^{p-1}(F))$,
which we can use together with the five lemma to obtain the complex we want for calculating $IH_{p-1}^*(X)$ in the proof 
of Theorem \ref{exthodge}.  

To do this, we will use $x^{(f/2)-p - \epsilon}L^2\Omega_{con}^*(M,g_{fc})$ in our long exact sequence to calculate
$IH^*_p(X)$, and we want to prove the following lemma.

\begin{lemma}\label{lemma6.3}
The cohomology of the complex:
\[
x^{(f/2)-p - \epsilon}L^2\Omega_{0}^*(M,g_{fc}):=\{ \omega \in x^{(f/2)-p - \epsilon}L^2\Omega_{con}^*(M,g_{fc}) \qquad \qquad
\]
\[
\hspace{2cm} \mid \lim_{s \to 0} \Pi_{p-1,p}\circ i^*_s \omega =0, \,  \lim_{s \to 0} \Pi_{p-1,p}\circ i^*_s d\omega=0\}
\]
is isomorphic to $IH_{p-1}^*(X)$.  Furthermore, we have the following long exact sequence on cohomology:
\[
\cdots \to H^{j-p-1}(B,H^p(F)) \stackrel{\delta}{\longrightarrow} IH_{p-1}^j(X) \stackrel{inc^*}{\longrightarrow} IH_{p}^j(X) 
\stackrel{r}{\longrightarrow} H^{j-p}(B,H^p(F)) \to \cdots,
\]
where $r=\lim_{s \to 0} \Pi_{p-1,p}\circ i^*_s$.
\end{lemma}

\begin{proof}
First we need to check that the map on cohomology induced by \[
r=\lim_{s \to 0} \Pi_{p-1,p}\circ i^*_s\] makes sense.
The proof is similar to the corresponding proof for when we replaced the absolute complex by the $L^2$ complex
in the previous subsection.
Note that on the end, an element $\omega \in x^{(f/2)-p - \epsilon}L^2\Omega_{con}^j(M,g_{fc})$ can be written as:
\[
\omega = \sum_i \omega_i := \sum_i \omega_{t,i} + \frac{dx}{x} \omega_{n,i-1},
\]
where $\omega_{t,i}$ has fibre degree $i$ and base degree $j-i$ and $\omega_{n,i-1}$ has fibre degree $i-1$ and base degree $j-i$.
By the Leray spectral sequence for $Y \to B$, when we decompose a closed $k$-form
$\alpha$ on $Y$ by fibre and base bidegree, each component is closed.  So if $d\omega=0$, so is
$d(\omega_p)$.  Now we can apply the same chain homotopy $R_{t,s}$ argument from
before to $\omega_p$ and we get that 
\[
\Pi_{p-1,p}\circ i^*_s \omega - \Pi_{p-1,p}\circ i^*_t \omega= i^*_s \omega_p - i^*_t \omega_p = d_Y i^*(R_{t,s}\omega_p).
\]
As before, by conormality, if $\omega \in x^{(f/2)-p - \epsilon}L^2\Omega_{con}^j(M,g_{fc})$, so is
$R_{t,s}\omega_p$, so we can conclude that on cohomology classes, $\Pi_{p-1,p}\circ i^*_s$ is independent of $s$,
which means that the map $\lim_{s \to 0} \Pi_{p-1,p}\circ i^*_s: IH^*_p(X) \to H^*(B,H^{p}(F))$ makes sense
when we calculate the first space using the weighted $L^2$ complex.

Now to check exactness at $IH_p^*(X)$, we need first to show that an element of the kernel of $\lim_{s \to 0} \Pi_{p-1,p}\circ i^*_s$
is in the image of the inclusion map.  We can apply a similar argument as in the proof of Lemma \ref{rel0cx}, 
but now just to the piece $\omega_p$.  If $\lim_{s \to 0} \Pi_{p-1,p}\circ i^*_s[\omega]=0$, this means that 
$i^*_s\omega_{t,p} = d_Y\alpha$.  Since $\omega_{t,p}$ has bidegree $(j-p,p)$, we get that $\alpha=\alpha_{p-1} + \alpha_p$,
where by conormality we may assume these are $\mathcal{O}(x^{-\epsilon})$ in norm.  But for $j\leq p$, such forms
are in $x^{(f/2)-p - \epsilon}L^2\Omega_{con}^*(M,g_{fc})$, thus $[\omega] = [\omega - d\chi \alpha] \in IH^j_p(X)$.
Showing that the image of the inclusion map is in the kernel follows the identical argument as in the proof of Lemma \ref{rel0cx}.

Finally, to check exactness at the cohomology of the complex $x^{(f/2)-p - \epsilon}L^2\Omega_{0}^*(M,g_{fc})$, we again need
to understand the boundary map.  The boundary map in this setting will take a closed form 
$
\alpha \in \Omega_{c\tau, p-1,p}^j(Y)
$
to $[d\chi \tilde{\alpha}]$ for any extension $\tilde{\alpha} \in x^{(f/2)-p - \epsilon}L^2\Omega_{con}^*(M,g_{fc})$.
Recall that $\alpha \in \Omega_{c\tau, p-1,p}^j(Y)$ is of the form
\[
\alpha_{p-1} + d_F\beta_{p-1}.
\]
This means that on the end, the cutoff constant extension 
\[
d\chi \alpha = d_Y \alpha = d_F \alpha_{p-1} + d_B \alpha_{p-1} + d_Bd_F \beta_{p-1} 
\]
so $d\chi \alpha$ has a piece of fibre degree $p$, which by construction is orthogonal to the kernel of $\delta_F$ by the Hodge
decomposition on fibres.  Thus this part is in the kernel of $\Pi_{p-1,p}\circ i^*_s$ for all $s$.  The $p-1$ piece 
is automatically in the kernel of $\Pi_{p-1,p}\circ i^*_s$, and both are $\mathcal{O}(1)$, thus in particular, 
$\mathcal{O}(x^{-\epsilon})$, as required to be in the correct weighted conormal $L^2$ space.  Thus the boundary
map makes sense.  The kernel of the inclusion is in the image of the boundary map by the same argument as in 
the proof of Lemma \ref{rel0cx},  Finally, the image of the boundary map is in the kernel of the inclusion because
$\chi\alpha \in x^{(f/2)-p - \epsilon}L^2\Omega_{con}^*(M,g_{fc})$.

\end{proof}

\subsection{End of proof}
The first and last vertical isomorphisms in the diagram in Theorem \ref{les} are standard maps from harmonic
forms on a compact manifold to cohomology, but in this flat setting, we can restrict those standard maps to
forms with fibre degree $p=(f/2)-c$.  
From the proof of Theorem \ref{exthodge}, we already have the third vertical map, and we know that the last
square is commutative.  

To construct the required vertical map, we use the complex
$x^{(f/2)-p - \epsilon}L^2\Omega_{0}^*(M,g_{fc})$ from Lemma \ref{lemma6.3}
to calculate $IH^j_{(f/2)-c-\epsilon}(X,B)$.
Let 
\[\omega = d\tau + \gamma \in
d(null(\Delta_{fc,c}^{j+1})) + \calH_{L^2}^j(M,g_{fc},c).
\]
By definition, $\omega \in x^{c-\epsilon}L^2\Omega^j(M,g_{fc})$, and by Lemma \ref{extclosed}, $d\omega=0$.
We also know from Equation \ref{harmend} and the sentences following it that on the end,
\[
\omega = \frac{dx}{x} \wedge u_{11} + \omega',
\]
where $u_{11} \in \calH^{j-(f/2)+c}(B,\calH^{(f/2)-c}(F))$ is constant in $x$ and $\omega' \in x^{c+\epsilon}L^2\Omega^j(M,g_{fc})$ and is conormal.  This means that
$r(\omega)= r(\omega')$, since the pullback of any form containing a $dx$ vanishes.  Now considering
what conormal forms in $x^{c+\epsilon}L^2\Omega^j(M,g_{fc})$ look like, we find
\[
\omega' = x^{c+\epsilon -(f/2)} \sum_{i=0}^f \frac{dx}{x} \wedge x^i \omega_{1,i}' + x^i \omega_{2,i}',
\]
where each family $\omega_{k,i}'$ is bounded in $x$.  Thus
\[
r(\omega') = r\left(x^{c+\epsilon-(f/2)} x^{f/2-c} \omega_{2,(f/2)-c} \right)= r\left(x^\epsilon \omega_{2,(f/2)-c} \right)=0,
\]
as required.  Thus elements of $\mathcal{IH}^j_{(f/2)-c-\epsilon}(M,g_{fc})$ are naturally cocycles in the 
complex

\noindent
$x^{(f/2)-p - \epsilon}L^2\Omega_{0}^*(M,g_{fc})$.  

The next step is to show that this map is injective.  Because we already know that the two spaces are isomorphic
(and finite dimensional), this will imply the map is also surjective.  So suppose that $\omega$ from above
is equal to $d\eta$ for $\eta \in x^{(f/2)-p - \epsilon}L^2\Omega_{0}^{j-1}(M,g_{fc})$.  Then on the end, 
\[
\eta = \frac{dx}{x} \wedge \eta_0 + \eta',
\]
where $\eta' \in x^{(f/2)-p + \epsilon}L^2\Omega_{0}^{j-1}(M,g_{fc})$.  The fact that $d\eta = \omega$ means that 
on the end, setting leading order terms in $x$ equal, we get
\[
-d_Y \eta_0 = u_{11} + u',
\]
where $u' \in x^{(f/2)-p + \epsilon}L^2\Omega_{0}^{j-1}(M,g_{fc})$.  But $u_{11}$ is harmonic on $Y$, thus
orthogonal to $d_Y \eta_0(x)$ for all $x$ and constant.  This means actually that 
$u_{11}=0$ and $\omega \in x^{c+\epsilon}L^2\Omega^j(M,g_{fc})$.  Now we can see that
\[
||\omega||_c^2 = <\omega, d\eta>_c =0,
\]
so $\omega=0$ as required.

We define the middle map in the top row of the diagram as
\[
inc^*:d(null(\Delta_{fc,c}^{j+1})) + \calH_{L^2}^j(M,g_{fc},c) \to \delta_{fc,c}(null(\Delta_{fc,c}^{j+1})) + \calH_{L^2}^j(M,g_{fc},c)
\]
by $inc^*(d\tau + \gamma) = \gamma$.  Thus the image of this map
is $\calH_{L^2}^j(M,g_{fc},c)$.  We know from Theorem \ref{l2hodge} that 
\[
 \calH_{L^2}^j(M,g_{fc},c) \cong {\rm Im}\left(IH^j_{(f/2)-c-\epsilon}(X,B) \to IH^j_{(f/2)-c+\epsilon}(X,B)\right),
 \]
so the middle square in the diagram commutes as desired.

To finish the proof we define $\overline{\partial}$ to be the map that makes the first square of the
diagram commute--we will identify it more directly later in Theorem \ref{1.6}.

\section{Proof of Theorem \ref{1.6}}
The first part of Theorem \ref{1.6} is proved above in Lemma \ref{PD}.  So it remains to define the boundary
data maps and prove the properties of their images.  To define the boundary data maps, recall that 
by definition,
\[
\mathcal{IH}^j_{(f/2)-c+\epsilon}(M,g_{fc}) := \delta_{fc,c}(null(\Delta_{fc,c}^{j+1})) + \calH_{L^2}^j(M,g_{fc},c).
\]
We know from Equation \ref{harmend} and the lines after it that on the end, an element of this space
has the form
\[
\omega= u_{12} + \omega',
\]
where $u_{12} \in \calH^{j-(f/2)+c}(B, \calH^{(f/2)-c}(F))$ and $\omega' \in x^{c+\epsilon}L^2\Omega^j(M,g_{fc})$.
We define the boundary data map on $\mathcal{IH}^j_{(f/2)-c+\epsilon}(M,g_{fc})$ by
\[
{\rm bd}_{T_c} \omega := r_c(\omega)=u_{12}.
\]

Also by definition, 
\[
\mathcal{IH}^j_{(f/2)-c-\epsilon}(M,g_{fc}) := d(null(\Delta_{fc,c}^{j+1})) + \calH_{L^2}^j(M,g_{fc},c).
\]
We know from Equation \ref{harmend} and the lines after it that on the end, an element of this space
has the form
\[
\omega= \frac{dx}{x}\wedge u_{11} + \omega',
\]
where $u_{11} \in \calH^{j-(f/2)+c}(B, \calH^{(f/2)-c}(F))$ and $\omega' \in x^{c+\epsilon}L^2\Omega^j(M,g_{fc})$.
We also have that
\[
*_c:\mathcal{IH}^j_{(f/2)-c-\epsilon}(M,g_{fc}) \to \mathcal{IH}^j_{(f/2)+c+\epsilon}(M,g_{fc})
\]
is an isomorphism.  Thus we define the boundary data map on this space by 
\[
{\rm bd}_{S_c} \omega := *_F {\rm bd}_{T_c} *_c \omega = u_{11}.
\]

In fact, both of these boundary maps come from the boundary data map
\[
BD:\calH_{ext}^*(M,g_{fc},c) \to \left[\calH^{*}(B, \calH^{(f/2)-c}(F))\right]^2,
\]\[ BD(\omega) = (\omega_1, \omega_2)
\]
for $\omega$ which on the end has the form
\[
\omega = \omega_1 + \frac{dx}{x}\wedge \omega_2 + \omega'.
\]

We want to show that the images of ${\rm bd}_{T_c}$ and ${\rm bd}_{S_c}$
are complementary orthogonal subspaces of $\calH^*(B, \calH^{(f/2)-c}(F))$
and that when $c=0$, they are complementary Lagrangian subspaces for the restriction of the
intersection pairing on $H^*(Y)$.

To do this, we use a variation of the boundary pairing from Section 6.1 of \cite{Me-aps}.  We start with
the operator $d + \delta_{fc,c}$ considered as acting on forms over $M$ with values in a topological
trivial coefficient line bundle $L$ with metric carrying the weight $x^{-c}$, that is so that 
$\omega \in x^c L^2\Omega^*(M,g_{fc})$ if and only
if $\omega \in L^2\Omega^*(M,g_{fc};L)$.  Then $d + \delta_{fc,c}$ is its own formal adjoint with respect to the $L^2$ 
inner product.  Now we consider the b-operator $P_c$ obtained by restricting $d + \delta_{fc,c}$ to the space of forms
over the end $Y\times [0,1)$ whose vertical part takes values in $\calH^{(f/2)-c}(F)$.  Note that when 
$g_{fc}$ is geometrically flat, $d + \delta_{fc,c}$ preserves this space of fibre harmonic forms, so 
this makes sense.  We end up with a self-adjoint b-operator, $P_c$ acting on 
$L^2\Omega^*(B\times[0,1),g_{b};\calH^{(f/2)-c}(F)\otimes L)$,
where the bundle $\calH^{(f/2)-c}(F)\otimes L$ carries a weight of $x^{f/2}$ in its metric.

Now we can apply the theorems from \cite{Me-aps} to $P_c$.  The formal null space associated to
the indicial root $\frac{f}{2}$, which is the critical root for $\omega \in L^2\Omega^*(M,g_{fc};L)$, 
is the space 
\[
F(P_c,\frac{f}{2}) = \{\omega_1 + \frac{dx}{x} \wedge \omega_2 \mid \omega_i \in \calH^{*}(B, \calH^{(f/2)-c}(F))\}.
\]
By Proposition 6.2 of \cite{Me-aps}, we can define a nondegenerate bilinear form $iB(u,v)$ on $F(P_c,\frac{f}{2})$ 
by
\[
iB(u,v) := \int_{B\times[0,1)} P_c(\chi u) \wedge *_c \chi v - \chi u \wedge *_c P_c(\chi v),
\]
where $\chi$ is a cutoff function supported near $x=0$ and $*_c$ is the weighted Hodge star operator from $M$
restricted to the subspace of $f/2-c$ degree fibre harmonic forms.  Furthermore, $iB$ is independent of the choice of $\chi$.
When we carry out integration by parts here, we find that 
\[
iB(u,v)=<u_1,v_2>_Y - <u_2,v_1>_Y.
\]
From this formula, we can see that this is a symplectic pairing on $F(P_c,\frac{f}{2}) \cong \left[\calH^{*}(B, \calH^{(f/2)-c}(F))\right]^2$.
\begin{lemma}\label{lagrangian}
The subspace ${\rm Im}(BD) \subset  F(P_c,\frac{f}{2})$
is a Lagrangian subspace under this pairing.
\end{lemma}

\begin{proof}
First, we can note that Theorem \ref{les} implies that
the dimension of ${\rm Im}(BD) \subset  F(P_c,\frac{f}{2})$ is equal to the dimension of $\calH^{*}(B, \calH^{(f/2)-c}(F))$.
Thus the subspace is half-dimensional.  We need to show that it is self-annihilating under the pairing $iB$.
This follows from Lemma 6.4 from \cite{Me-aps} by the same argument, with a bit of adaptation.  
First suppose that $u,v \in {\rm Im}(BD) \subset  F(P_c,\frac{f}{2})$.  Then there exist extended harmonic forms
$\tilde{u}, \tilde{v}$ such that, 
\[
\tilde{u} = \chi u + u' \qquad \tilde{v}=\chi v + v', \qquad u',v' \in x^{c+\epsilon}L^2\Omega^j(M,g_{fc}),
\]
so $(d+\delta_{fc,c})(\chi u) =-(d+\delta_{fc,c}) u'$ and similarly for $v$.
Then 
\begin{eqnarray*}
iB(u,v) &=& \int_{B\times[0,1)} P_c(\chi u) \wedge *_c \chi v - \chi u \wedge *_c P_c(\chi v)\\
&=& \int_{Y\times[0,1)} (d+\delta_{fc,c})(\chi u) \wedge *_c \chi v - \chi u \wedge *_c (d+\delta_{fc,c})(\chi v)\\
&=& \int_{Y\times[0,1)} -(d+\delta_{fc,c})(u') \wedge *_c \chi v - \chi u \wedge *_c (d+\delta_{fc,c})(\chi v).
\end{eqnarray*}
Because all of the terms have support near $x=1$, we can replace the integral here with 
an integral over all of $M$.  Then in the first term, by the decay of $u'$, we can integrate by parts using the extended 
$L^2$ pairing to get
\begin{eqnarray*}
&=& \int_{M} -(u') \wedge *_c (d+\delta_{fc,c})\chi v - \chi u \wedge *_c (d+\delta_{fc,c})(\chi v) \\
&=& \int_{M} \tilde{u}\wedge *_c (d+\delta_{fc,c})\chi v \\
&=& \int_{M} \tilde{u}\wedge *_c (d+\delta_{fc,c})v' =0.
\end{eqnarray*}
Where in the last step, we have used the vanishing of $v'$ to justify the integration by parts using the 
extended $L^2$ pairing.
\end{proof}

Now consider $({\rm Im}({\rm bd}_{T_c}),0)$ and $(0, {\rm Im}({\rm bd}_{S_c}))$. We know by Lemma \ref{4.4}
that these images are complementary subspaces of ${\rm Im}(BD)$, so their dimensions must add up
to the dimension of $\calH^{*}(B, \calH^{(f/2)-c}(F))$.   By Lemma \ref{lagrangian}, 
we also know that for $u_1 \in {\rm Im}({\rm bd}_{T_c})$ and $v_2 \in {\rm Im}({\rm bd}_{S_c})$,
\[
0 = iB((u_1,0), (0,v_2)) = <u_1,v_2>_Y,
\]
so they are also orthogonal.  This proves the second part of Theorem \ref{1.6}.
Finally, we see from the first part of the theorem that when $c=0$, the Hodge star $*_Y$, is an isomorphism between these images, 
so they are of equal dimension, and the intersection pairing on $Y$ is just $<u, *_Y v>$, which
proves the final part of the theorem.

\end{document}